\documentclass[11pt]{amsart}

\usepackage{latexsym}
\usepackage{amssymb}
\usepackage{amsmath}
\usepackage{bm}

\usepackage{mathrsfs}
\usepackage[scr=rsfs,cal=boondox]{mathalfa}

\newtheorem{theorem}{Theorem}[section]
\newtheorem{lemma}[theorem]{Lemma}
\newtheorem{proposition}[theorem]{Proposition}
\newtheorem{corollary}[theorem]{Corollary}

\numberwithin{equation}{section}

\theoremstyle{definition}
\newtheorem{definition}[theorem]{Definition}

\newtheorem{example}[theorem]{Example}

\theoremstyle{remark}

\def\N{\mathbb{N}}
\def\Z{\mathbb{Z}}
\def\Q{\mathbb{Q}}
\def\R{\mathbb{R}}

\def\z{\mathbf{z}}

\def\U{\mathcal{U}}
\def\V{\mathcal{V}}
\def\W{\mathcal{W}}

\def\lk{\textcircled{$\star$}_{\ell,k}}

\def\ostar{\textcircled{$*$}}

\def\FP{\text{FP}}

\begin{document}

\title[Infinite monochromatic patterns in $\Z$]
{Infinite monochromatic patterns 
\\
in the integers}

\author{Mauro Di Nasso}

\address{Dipartimento di Matematica\\
Universit\`a di Pisa, Italy}

\email{mauro.di.nasso@unipi.it}

\thanks{This work was supported 
by the Italian national research project PRIN 2017 
``Mathematical logic: models, sets, computability"}

\subjclass[2000]
{Primary 05D10, 11B75; Secondary 03E05.}

\keywords{Ramsey Theory, Monochromatic patterns in the integers, 
Algebra in the Stone-\v{Cech} compactification}


\begin{abstract}
We show the existence of several infinite monochromatic patterns 
in the integers obtained as values of suitable symmetric polynomials;
in particular, we obtain extensions of both the additive 
and multiplicative versions of Hindman's theorem. 
These configurations are obtained by means of suitable symmetric 
polynomials that mix the two operations. The simplest example is the following.
For every finite coloring $\N=C_1\cup\ldots\cup C_r$ there exists
an infinite increasing sequence $a<b<c<\ldots$ such that all elements
below are monochromatic:
$$a,b,c,\ldots, a+b+ab, a+c+ac, b+c+bc,\ldots,a+b+c+ab+ac+bc+abc,\ldots.$$
\noindent
The proofs use tools from algebra in the space of ultrafilters $\beta\Z$.
\end{abstract}

\maketitle


\maketitle

\section{Introduction}

Many of the classic results in arithmetic Ramsey Theory are about the existence of  
monochromatic patterns found in any given finite coloring of the integers
or of the natural numbers. 
(As usual in Ramsey Theory, ``coloring" means
partition, and a set is called ``monochromatic" if it is included in one piece of the partition).
A great amount of work has been devoted to the search for monochromatic finite patterns, 
the archetype of which are the (finite) arithmetic progressions. Indeed,
a cornerstone in this field of research is \emph{Van der Waerden Theorem},
stating that in any finite coloring of the natural numbers one always finds
arbitrarily long monochromatic arithmetic progressions.

Also infinite patterns have been repeatedly considered by researchers, 
although for them the variety of relevant examples does not seem to be comparable 
to that of finite configurations. 
The prototype of infinite monochromatic configurations in the natural numbers
is the one given by the celebrated \emph{Hindman Theorem}:
``For every finite coloring $\N=C_1\cup\ldots\cup C_r$ of the natural numbers, 
there exists an injective sequence $(x_n)_{n=1}^\infty$ such that all finite sums
$x_{n_1}+\ldots+x_{n_s}$ where $n_1<\ldots<n_s$ are monochromatic."
The same result holds for the natural numbers with multiplication, and more generally, 
for any cancellative semigroup.

Generalizations of Hindman's Finite Sum Theorem are obtained as corollaries 
of \emph{Milliken-Taylor Theorem}; for instance, for every choice
of coefficients $a_1,\ldots,a_m\in\N$, there exists an injective sequence 
$(x_n)_{n=1}^\infty$ such that all sums 
$a_1 \sum_{n\in F_1}x_n+\ldots+a_m\sum_{n\in F_m}x_n$ 
where the nonempty finite sets $F_1<\ldots<F_m$ are arranged in increasing order
(that is, $\max F_i<\min F_{i+1}$), are monochromatic.
In the recent papers \cite{bhw,lu}, within the general framework of semigroups,
polynomial extensions of Milliken-Taylor Theorem
have been proved which produce plenty of similar (but much more general)
infinite monochromatic patterns.

The goal of this paper is to show that several infinite monochromatic configurations
on the integers and on the natural numbers can be found where the additive 
and the multiplicative structure are mixed with the use of symmetric polynomials. 
(We pay attention that the considered patterns be not degenerate,
in the sense that they are made of pairwise distinct elements.)
To this end, we consider a class of associative and commutative operations
on the integers originated by affine transformations, and then use the machinery 
of algebra on the Stone-\v{C}ech compactification.

The following property is probably the simplest corollary of our results 
which already provides a significant example of the type of
``symmetrical" monochromatic patterns that can be obtained
combining the sum and product operations
(see Example \ref{example1} with $\ell=1$):

\begin{itemize}
\item
\emph{For every finite coloring of the natural numbers
there exists an injective sequence $(x_n)_{n=1}^\infty$ such that
all symmetric expressions below are monochromatic:}
\begin{multline*}
x_1,\ x_2,\ x_3,\ \ldots\ ,\ 
x_1+x_2 + x_1 x_2,\ x_1+x_3+x_1x_3,\ x_2+x_3 +x_2 x_3,\ \ldots\ ,
\\
x_1+x_2+x_3+x_1 x_2+x_1 x_3+x_2 x_3+x_1 x_2 x_3,\ \ldots.
\end{multline*}
\end{itemize}
(For the sake of brevity, we listed explicitly only the expressions
that involve the first three elements of the sequence.)
Notice that the above pattern is obtained by considering
finite iterations of the symmetric polynomial function $P(a,b)=a+b+ab$.
The key observation for the proof is that such a function is an associative operation on $\N$,
and in fact is the operation inherited from the multiplicative structure 
via the affine transformation $T:a\mapsto a+1$.

In this regard, it is worth mentioning that monochromatic patterns in the natural numbers that
mix additive and multiplicative structure are of great interest in the current research in 
arithmetic Ramsey Theory.
For instance, it was only in 2010 that V. Bergelson \cite{be} and N. Hindman \cite{hi2} 
independently proved that the configuration $\{a,b,c,d\}$ where $a+b=c\cdot d$ is monochromatic.
In 2017 by J. Moreira \cite{mo} showed 
that the pattern $\{a,a+b,a\cdot b\}$ is monochromatic.
In 2019, J.M. Barrett, M. Lupini and J. Moreira \cite{blm}, building also on
previous work by Luperi Baglini and the author \cite{dl}, 
proved other similar partition regular configurations,
including $\{a, a+b, a+b+a\cdot b\}$. 
It is still an open problem whether $\{a,b,a+b,a\cdot b\}$
is a monochromatic configuration.

The paper is organized as follows.
In Sections 2 and 3, we present our results and give several examples.
In Section 4 we recall all the notions required for the proofs, which are given in
the following Section 5. The last Section 6 contains a list of remarks
and possible directions for future reserach.

%

\begin{section}{Symmetric polynomials and monochromatic configurations}\label{sec-results}

Throughout the paper, we denote by $\N=\{1,2,\ldots\}$ the set of 
positive integers.

The combinatorial configurations we are interested in are symmetric,
in the sense that they originate from suitable symmetric polynomials.
Recall the following

\begin{definition}
For $j=1,\ldots,n$, the \emph{elementary symmetric polynomial}
in $n$ variables is the polynomial:
$$e_{j}(X_1,\ldots,X_n)\ =\sum_{1\le i_1<\ldots<i_j\le n}X_{i_1}\cdots X_{i_j}\ =
\sum_{G\in[\{1,\ldots,n\}]^j}\prod_{s\in G}X_s$$
where we used the notation $[X]^j=\{G\subseteq X\mid |G|=j\}$.
\end{definition}

Notice that for all real numbers $a_1,\ldots,a_n$, 
we have 
$$\prod_{j=1}^n(a_j+1)\ =\ c+1$$
where
$$c\ =\ \sum_{j=1}^{n}e_{j}(a_1,\ldots,a_n)\ =
\sum_{\emptyset\ne G\subseteq\{1,\ldots,n\}}\prod_{s\in G}a_s.$$
More generally, for $\ell, k\ne 0$, it is easily verified that
$$\prod_{j=1}^n(\ell a_j+k)\ =\ \ell c+k$$
where
\begin{multline*}
c\ =\ \sum_{j=1}^{n}\ell^{j-1}k^{n-j}e_{j}(a_1,\ldots,a_n)+\frac{k^n-k}{\ell}\ =
\\
=\ \sum_{\emptyset\ne G\subseteq\{1,\ldots,n\}}
\left(\ell^{|G|-1}k^{n-|G|}\cdot\prod_{s\in G}a_s\right)+ 
\frac{k^n-k}{\ell}.
\end{multline*}

The crucial point here is the fact that there exists
a commutative and associative operation $\lk$ such that
$a_1\,\lk\,\cdots\,\lk\, a_n=c$,
where $c$ is the number defined as above. (See \S \ref{operationslk}).

Notice that the above number $c=c(a_1,\ldots,a_n)$ belongs
to $\Z$ for all $a_1,\ldots,a_n\in\Z$ if and only if $\ell$ divides $k(k-1)$.
This justifies our attention on the following class of symmetric polynomials.

\begin{definition}\label{def-symmetricpolynomial}
For $\ell,k\in\Z$ with $\ell,k\ne 0$, 
the \emph{$(\ell,k)$-symmetric polynomial} in $n$ variables is:
\begin{multline*}
\mathfrak{S}_{\ell,k}(X_1,\ldots,X_n)\ :=\ 
\sum_{j=1}^{n} \ell^{j-1} k^{n-j} e_j(X_1,\ldots,X_n) + \frac{k^n-k}{\ell}
\\
=\ \sum_{\emptyset\ne G\subseteq\{1,\ldots,n\}}
\left(\ell^{|G|-1}k^{n-|G|}\cdot\prod_{s\in G}X_s\right)+ 
\frac{k^n-k}{\ell}.
\end{multline*}
\end{definition}


For instance, if $k=1$ and $n=4$, then for every $\ell\ne 0$:
\begin{multline*}
\mathfrak{S}_{\ell,1}(a,b,c,d) \ = \
a+b+c+d+\ell(ab+ac+ad+bc+bd+cd)+ 
\\
+\ell^2(abc+abd+acd+bcd)+
\ell^3 abcd.
\end{multline*}


Recall that for infinite sequences of natural numbers $(x_n)_{n=1}^\infty$,
the corresponding set of finite sums is the set:
$$\text{FS}(x_n)_{n=1}^\infty\ :=\ \{x_{n_1}+\ldots+x_{n_s}\mid n_1<\ldots<n_s\}.$$

A cornerstone result in arithmetic Ramsey Theory shows
the existence of infinite monochromatic patterns of finite sums.

\begin{itemize}
\item
Hindman Finite Sums Theorem (1974) \cite{hi1}:
\emph{For every finite coloring $\N=C_1\cup\ldots\cup C_r$
there exist a color $C_i$ and an injective sequence $(x_n)_{n=1}^\infty$ such that
$\text{FS}(x_n)_{n=1}^\infty\subseteq C_i$.
More generally, for every injective sequence of natural numbers
$(x_n)_{n=1}^\infty$ and for every finite coloring 
$\text{FS}(x_n)_{n=1}^\infty=C_1\cup\ldots\cup C_r$
of the corresponding set of finite sums,
there exist an injective sequence $(y_n)_{n=1}^\infty$ and
a color $C_i$ such that $\text{FS}(y_n)_{n=1}^\infty\subseteq C_i$.}
\end{itemize}

The same result is also true if one considers finite products
instead of finite sums.

In analogy with the set of finite sums
we give the following

\begin{definition}\label{def-symmetricsystem}
Let $(x_n)_{n=1}^\infty$ be an infinite sequence, and let $\ell,k\in\Z$ with $\ell,k\ne 0$. 
The corresponding
$(\ell,k)$-\emph{symmetric system} is the set:
$$\mathfrak{S}_{\ell,k}(x_n)_{n=1}^\infty\ :=\ 
\left\{\mathfrak{S}_{\ell,k}(x_{n_1},\ldots,x_{n_s})\mid n_1<\ldots<n_s\right\}.$$
\end{definition}

For suitable $\ell$ and $k$, $(\ell,k)$-symmetric systems are partition regular
on $\Z$ and on $\N$.

\begin{theorem}\label{main1}
Assume that $\ell,k\ne 0$ are integers where $\ell$ divides $k(k-1)$.
Then for every finite coloring $\Z=C_1\cup\ldots\cup C_r$ 
there exist an injective sequence $(x_n)_{n=1}^\infty$ of integers and
a color $C_i$ such that $\mathfrak{S}_{\ell,k}(x_n)_{n=1}^\infty\subseteq C_i$.

More generally, for every injective sequence of integers $(x_n)_{n=1}^\infty$
and for every finite coloring 
$\mathfrak{S}_{\ell,k}(x_n)_{n=1}^\infty=C_1\cup\ldots\cup C_r$
of the corresponding $(\ell,k)$-symmetric system,
there exist an injective sequence $(y_n)_{n=1}^\infty$ of integers and
a color $C_i$ such that $\mathfrak{S}_{\ell,k}(y_n)_{n=1}^\infty\subseteq C_i$.

Moreover, for positive $\ell\in\N$, the above partition regularity properties
are also true if we replace the integers $\Z$ with the natural numbers $\N$.
\end{theorem}

Here are two of the simplest examples.

\begin{example}\label{example1}
When $k=1$,  for every $\ell\in\N$ one obtains the following infinite monochromatic pattern in
the natural numbers, where the sequence $(x_n)_{n=1}^\infty$ is injective:\footnote
{~Following the common use, for simplicity we will say that: ``the pattern (or configuration)
$\mathcal{S}(x_n)_{n=1}^\infty$ 
is monochromatic in $X$" to mean that: ``for every finite partition $X=C_1\cup\ldots\cup C_r$ there
exist a color $C_i$ and a sequence $(x_n)_{n=1}^\infty$ such that 
$\mathcal{S}(x_n)_{n=1}^\infty\subseteq C$."}
$$\left\{\sum_{\emptyset\ne G\subseteq F}
\left(\ell^{|G|-1}\prod_{s\in G}x_s\right)\,\Bigg|\ \emptyset\ne F\subset\N\ \text{finite}\right\}.$$

That is, the following elements are monochromatic:
\begin{itemize}
\item
$x_s$ for all $s$,
\item
$x_s+x_t + \ell\, x_s x_t$ for all $s<t$,
\item
$x_s+x_t+x_u+\ell(x_s x_t+ x_s x_u+ x_t x_u)+\ell^2\, x_s x_t x_u$
for all $s<t<u$, 
\item
$x_s+x_t+x_u+x_v+\ell(x_s x_t+x_s x_u+x_s x_v+
x_t x_u+x_t x_v+x_u x_v)+\ell^2(x_s x_t x_u+
 x_s x_t x_v+ x_s x_u x_v+ x_t x_u x_v)+\ell^3\, x_s x_t x_u x_v$ for all 
$s<t<u<v$; and so forth.
\end{itemize}
\end{example}

\begin{example}
When $\ell=k=2$, 
one obtains the following infinite monochromatic pattern in
the natural numbers, where the sequence $(x_n)_{n=1}^\infty$ is injective:
$$\left\{2^{|F|-1}\cdot\left(\sum_{\emptyset\ne G\subseteq F}
\prod_{s\in G}x_s\right)+2^{|F|-1}-1\,\Bigg|\ \emptyset\ne F\subset\N\ \text{finite}\right\}.$$
That is, the following elements are monochromatic:
\begin{itemize}
\item
$x_s$ for all $s$,
\item
$2\,(x_s+x_t + x_s x_t)+1$ for all $s<t$,
\item
$4\,(x_s+x_t+x_u+x_s x_t+x_s x_u+x_t x_u+x_s x_t x_u)+3$
for all $s<t<u$, 
\item
$8\,(x_s+x_t+x_u+x_v+x_s x_t+x_s x_u+ x_s x_v+
x_t x_u+ x_t x_v+x_u x_v+x_s x_t x_u+
x_s x_t x_v+ x_s x_u x_v+ x_t x_u x_v+ x_s x_t x_u x_v)+7$ for all 
$s<t<u<v$; and so forth.
\end{itemize}
\end{example}

As already mentioned, a fundamental result in arithmetic Ramsey Theory 
is the classic

\begin{itemize}
\item
Van der Waerden Theorem (1927)  \cite{VdW}:
\emph{For every finite coloring $\N=C_1\cup\ldots\cup C_r$
and for every $L\in\N$ there exists a monochromatic
arithmetic progression of length $L$; that is, there exist
a color $C_i$ and elements $a,b\in\N$ such that
$a, a+b, a+2b, \ldots, a+Lb\in C_i$.}
\end{itemize}

The following year 1928, the above Ramsey property was strengthened by Brauer \cite{br},
who proved that one can also have the common difference $b$ 
of the same color as the elements of the progression.

A few decades later, as a result of his studies about partition regularity
of homogeneous systems of linear Diophantine equations,
Deuber \cite{de} demonstrated further generalizations; in particular,
he showed the partition regularity of the so-called $(m,p,c)$-sets. 

\begin{itemize}
\item
Deuber Theorem (1974)  \cite{de}:
\emph{For every $m,p,c\in\N$ and for every
finite coloring $\N=C_1\cup\ldots\cup C_r$
there exists a monochromatic $(m,p,c)$-set;
that is, there exist a color $C_i$ and elements
$a_0,a_1,\ldots,a_m\in C_i$ such that
$a_j+\sum_{s=0}^{j-1}n_s a_s\in C_i$ for every
$j\in\{1,\ldots,m\}$ and for all $n_0,\ldots,n_{j-1}\in\{-p,\ldots,p\}$.}
\end{itemize}

The following analogue of Deuber Theorem holds in our context.

\begin{theorem}\label{main2}
Let $\ell,k$ be integers where $\ell\ne 0$ divides $k-1$,
let $m\in\N$, and let $L\in\N$. Then for every finite coloring $\Z=C_1\cup\ldots\cup C_r$ 
there exist a color $C_i$ and elements $a_0,a_1,\ldots,a_m\in C_i$
such that for every $j=1,\ldots,m$ and for all $n_0,\ldots,n_{j-1}\in\{0,1,\ldots,L\}$:
$$\frac{1}{\ell}\left((\ell a_j+k)\prod_{s=0}^{j-1}(\ell a_s+k)^{n_s}-k\right)\in C_i$$
where we can assume that $(\ell a_j+k)\ne 0,1,-1$ for all $j$.\footnote
{~This condition is needed to get
meaningful configurations.}

Moreover, for positive $\ell\in\N$, the above partition regularity property is
also true if we replace the integers $\Z$ with the natural numbers $\N$.
\end{theorem}

\begin{example}
In the simple case when $\ell=k=1$ and $m=L=2$
one obtains the following monochromatic pattern in the natural numbers:
\footnotesize{\begin{multline*}
a,\ b,\ c,\ a+b+ab,\ a+c+ac,\ b+c+bc,\ a+b+c+ab+ac+bc+abc,
\\
a^2b+a^2 +2ab+2a+b,\ 
a^2c+a^2 +2ac+2a+c,\ b^2c+b^2 +2bc+2b+c,
\\
a^2 bc+a^2b+a^2c+2abc+2ab+2ac+a^2+bc+2a+b+c,
\\
ab^2c+ab^2+b^2c+2abc+2ab+2bc+b^2+ac+a+2b+c,
\\
a^2b^2c+2a^2bc+2ab^2c+4abc+a^2c+b^2c+2ac+2bc+
a^2b^2+2a^2b+2ab^2+4ab+a^2+b^2+2a+2b+c.
\end{multline*}}
\end{example}

As a consequence of Theorem \ref{main2}, the following analogue of 
Brauer Theorem is proved, where elements in the monochromatic
configuration are all distinct:

\begin{theorem}\label{main3}
Let $\ell,k$ be integers where $\ell\ne 0$ divides $k-1$.
Then for every finite coloring $\Z=C_1\cup\ldots\cup C_r$ 
and for every $L\in\N$ there exist a color $C_i$ and elements $a,b$ such that
$$a,\ b,\  \frac{1}{\ell}\big((\ell a+k)(\ell b+k)-k\big),\ 
\ldots\ ,\  \frac{1}{\ell}\big((\ell a+k)(\ell b+k)^L-k\big)\ \in\ C_i$$
where we can assume the above elements to be pairwise distinct.

Moreover, for positive $\ell\in\N$, the above partition regularity property
is also true if we replace the integers $\Z$ with the natural numbers $\N$.
\end{theorem}

\begin{example}
In the simplest case when $\ell=k=1$, for every $L\in\N$
one obtains the following monochromatic pattern in the natural numbers,
where all elements are distinct:
\begin{multline*}
a,\ b,\ a+b+ab,\ ab^2+b^2+2ab+2b+a,
\\
ab^3+b^3+3ab^2+3b^2+3ab+3b+a,\ \ldots\ ,\ (a+1)(b+1)^L-1.
\end{multline*}
\end{example}

\medskip
\section{Symmetric patterns and Milliken-Taylor Theorem}

Below, we denote by $[\N]^m$ 
the family $=\{F\subset \N\mid |F|=m\}$ of all subsets of $\N$ of cardinality $m$;
and for nonempty finite $F,G\subset\N$ we
write $F<G$ to mean that $\max F<\min G$.

The following result, that was independently proved by
Milliken and Taylor soon after Hindman proved his Finite Sums Theorem, 
is a common strengthening of Hindman and Ramsey Theorems.

\begin{itemize}
\item
Milliken-Taylor Theorem (1975) \cite{mi,ta}:
\emph{For every finite coloring $[\N]^m=C_1\cup\ldots\cup C_r$
there exist an injective sequence $(x_n)_{n=1}^\infty$ of natural numbers
and a color $C_i$ such that
$$\left\{  \{ x_{F_1},\ldots,x_{F_m} \} \mid F_1<\ldots< F_m \right\}\subseteq C_i$$
where for $F=\{n_1<\ldots<n_s\}\subset\N$ we denoted
$x_F=x_{n_1}+\ldots+x_{n_s}$.}
\end{itemize}

Clearly, when $m=1$ one obtains Hindman Theorem;
and when all $F_j$ are singletons, one obtains Ramsey Theorem.

Similarly to Hindman Theorem and van der Waerden Theorem, also 
Milliken-Taylor Theorem has an analogue with our symmetric patterns.

For convenience, we now extend the definition of 
$(\ell,k)$-symmetric systems $\mathfrak{S}_{\ell,k}$
to cases where $k=0$ or $\ell=0$,
so as to also include the usual finite products and finite sums.

\begin{definition}\label{def-symmetricpolynomial2}
For integers $\ell\ne 0$ we set:
\begin{itemize}
\item
$\mathfrak{S}_{\ell,0}(a_1,\ldots,a_n)=a_1\,\textcircled{$\star$}_{\ell,0}\cdots\,\textcircled{$\star$}_{\ell,0}a_n=
\ell^{n-1}a_1\cdots a_n$.
\end{itemize}
We also set:
\begin{itemize}
\item
$\mathfrak{S}_{0,1}(a_1,\ldots,a_n)=a_1\,\textcircled{$\star$}_{0,1}\cdots\,\textcircled{$\star$}_{0,1}a_n=
a_1+\ldots+a_n$.
\end{itemize}
\end{definition}

%
%


\begin{theorem}\label{main4}
Assume that
\begin{itemize}
\item[$(a)$]
$(\ell_j,k_j)_{j=1}^m$ is a finite sequence of pairs of integers 
where for every $j=1,\ldots,m$,
either $\ell_j\ne 0$ divides $k_j(k_j-1)$ or $(\ell_j,k_j)=(0,1)$;
\item[$(b)$]
$f:\Z^m\to\Z$ is any function.
\end{itemize}
Then for every finite coloring $\Z=C_1\cup\ldots\cup C_r$
there exist injective sequences
$(x^{(j)}_n)_{n=1}^\infty$ for $j=1,\ldots,m$ and
there exists a color $C_i$ with the properties that
$(x^{(j)}_n)_{n=1}^\infty=(x^{(j')}_n)_{n=1}^\infty$
whenever $(\ell_j,k_j)=(\ell_{j'},k_{j'})$, and 
$$\left\{f\left(x^{(1)}_{F_1},\ldots,x^{(m)}_{F_m}\right)\,\Big|\, F_1<\ldots<F_m\right\}\ \subseteq\ C_i,$$
where for $F=\{n_1<\ldots<n_s\}\subset\N$ we denoted
$x^{(j)}_F=\mathfrak{S}_{\ell_j,k_j}(x^{(j)}_{n_1},\ldots,x^{(j)}_{n_s})$.

Moreover, if we also assume that all $\ell_j\ge 0$ and 
if $f:\N^m\to\Z$ satisfies the condition:
\begin{itemize}
\item[$(\dagger)$]
$\exists\,\overline{n}_1\ \forall n_1\ge\overline{n}_1\ 
\exists\,\overline{n}_2\ \forall n_2\ge\overline{n}_2\ \ldots\
\exists\,\overline{n}_m\ \forall n_m\ge\overline{n}_m$ one has
that $f(n_1,n_2,\ldots,n_m)\in\N$,
\end{itemize}
then the above partition regularity property
is also true if we replace the integers $\Z$ with the natural numbers $\N$.
\end{theorem}

Clearly, every function $f:\N^m\to\N$ 
trivially satisfies condition $(\dagger)$; however, we remark that there are
more relevant examples, including a large class of polynomial functions (see below).

\smallskip
Recall that a \emph{multi-index} is a tuple 
$\alpha=(\alpha_1,\ldots,\alpha_m)\in(\N\cup\{0\})^m$.
If $\z=(z_1,\ldots,z_m)$ is a vector of variables 
and $\alpha=(\alpha_1,\ldots,\alpha_m)$ is a multi-index, 
then we write $\z^\alpha$ to denote the monomial 
$\prod_{i=1}^m z_i^{\alpha_i}$.
Polynomials in the variables $z_1,\ldots,z_m$ are written
in the form $P(\z)=\sum_\alpha c_\alpha\z^\alpha$,
where $\alpha$ are multi-indexes
and where $c_\alpha$ are the coefficients of monomials $\z^\alpha$.
The \emph{support} of $P$ is the finite set
$\text{supp}(P)=\{\alpha\mid c_\alpha\ne 0\}$.
Now consider the \emph{anti-lexicographic order} on the multi-indexes,
where for $\alpha\ne\beta$ one sets:
$$(\alpha_1,\ldots,\alpha_m)\ <\ (\beta_1,\ldots,\beta_m)\ \Longleftrightarrow\ 
\alpha_i<\beta_i\ \text{where}\ i=\max\{j\mid \alpha_j\ne\beta_j\}.$$

The \emph{leading term} of a polynomial $P=\sum_\alpha c_\alpha\z^\alpha$
is the monomial $c_\alpha\z^\alpha$ where $\alpha=\max\text{Supp}(P)$
is the greatest multi-index of $P$ according to the anti-lexico\-graphic order.
The \emph{leading coefficient} of $P$ is the coefficient $c_\alpha$ of its leading term.

\begin{proposition}\label{polynomials1}
Let $P\in\Z[z_1,\ldots,z_m]$
be a polynomial in several variables over the integers
with positive leading coefficient. Then
the polynomial function $P(z_1,\ldots,z_m)$ satisfies condition $(\dagger)$ of
Theorem \ref{main4}.
\end{proposition}

\begin{proof}
It is a straightforward consequence of the following general property of polynomials,
restricted to variables that are natural numbers:
\begin{itemize}
\item
If $P\in\R[z_1,\ldots,z_m]$ has positive leading coefficient, then:
\\
$\exists\,\overline{x}_1\ \forall x_1\ge\overline{x}_1\
\exists\,\overline{x}_2\ \forall x_2\ge\overline{x}_2\ \, \ldots\ \,
\exists\,\overline{x}_m\ \forall x_m\ge\overline{x}_m$ 
one has that
$P(x_1,x_2,\ldots,x_m)>0$.
\end{itemize}

In the base case of a single variable, let $P(z)=\sum_{j=1}^d c_j z^j$ where
the leading coefficient $c_d>0$. Then $\lim_{x\to +\infty}P(x)=+\infty$,
and so there exists $\overline{x}_1$ such that $P(x_1)>0$ for all $x_1\ge\overline{x}_1$.

At the inductive step, let $P=\sum_\alpha c_\alpha\z^\alpha\in\R[z_1,\ldots,z_m,z_{m+1}]$,
where the leading term $c_\gamma\z^\gamma$ has positive coefficient $c_\gamma>0$.
If $\gamma=(\beta_1,\ldots,\beta_m,d)$, then
we can write $P=\sum_{j=0}^d P_j\cdot(z_{m+1})^j$ for suitable polynomials
$P_j\in\R[z_1,\ldots,z_m]$ for $j=0,\ldots,d$.
Notice that the leading term of $P_d$ is $c_\gamma\z^\beta$
where $\beta=(\beta_1,\ldots,\beta_m)$.
By the inductive hypothesis applied to $P_d$, we have that
$\exists\,\overline{x}_1\ \forall x_1\ge\overline{x}_1\
\exists\,\overline{x}_2\ \forall x_2\ge\overline{x}_2\ \, \ldots\ \,
\exists\,\overline{x}_m\ \forall x_m\ge\overline{x}_m$ 
one has that $P_d(x_1,x_2,\ldots,x_m)>0$.
Given any $x_1,\ldots,x_m$ as above, consider the polynomial
$Q(z):=P(x_1,\ldots,x_m,z)\in\R[z]$. Notice that
$Q(z)=\sum_{j=0}^d a_j z^j$
where $a_j:=P_j(x_1,x_2,\ldots,x_m)$.
Since the leading term $a_d=P_d(x_1,x_2,\ldots,x_m)$ is positive, 
$\lim_{x\to +\infty}Q(x)=+\infty$
and so there exists $\overline{x}_{m+1}$ such that for every 
$x_{m+1}\ge \overline{x}_{m+1}$
one has that $Q(x_{m+1})=P(x_1,\ldots,x_m,x_{m+1})>0$, 
as desired.
\end{proof}

For the natural numbers $\N$, we can prove a modified
version of the previous theorem where a smaller class of $(\ell_j,k_j)$
is allowed, but where a larger class of functions $f$ is considered.

\begin{theorem}\label{main5}
Assume that
\begin{itemize}
\item[$(a)$]
$(\ell_j,k_j)_{j=1}^m$ is a finite sequence of pairs of integers 
where  for every $j=1,\ldots,m$, either $\ell_j>0$ and $k_j\in\{0,1\}$,
or $(\ell_j,k_j)=(0,1)$;
\item[$(b)$]
$f$ is an $m$-variable function that satisfies the following property,
where all variables $n_j,\overline{n}_j,N_j\in\N$:
\begin{itemize}
\item[$(\ddagger)$]
$\exists\,\overline{n}_1,\!N_1\ \forall n_1\ge\overline{n}_1\ 
\exists\,\overline{n}_2,\!N_2\ \forall n_2\ge\overline{n}_2\ \ldots\
\exists\,\overline{n}_m,\!N_m\ \forall n_m\ge\overline{n}_m$
one has
that $f(n_1N_1,n_2 N_2,\ldots,n_m N_m)\in\N$.
\end{itemize}
\end{itemize}
Then for every finite coloring $\N=C_1\cup\ldots\cup C_r$
there exist injective sequences
$(x^{(j)}_n)_{n=1}^\infty$ for $j=1,\ldots,m$ and
and there exists a color $C_i$ with the properties that
$(x^{(j)}_n)_{n=1}^\infty=(x^{(j')}_n)_{n=1}^\infty$
whenever $(\ell_j,k_j)=(\ell_{j'},k_{j'})$, and
$$\left\{f\left(x^{(1)}_{F_1},\ldots,x^{(m)}_{F_m}\right)\,\Big|\, F_1<\ldots<F_m\right\}\ \subseteq\ C_i,$$
where for $F=\{n_1<\ldots<n_s\}\subset\N$ we denoted
$x^{(j)}_F=\mathfrak{S}_{\ell_j,k_j}(x^{(j)}_{n_1},\ldots,x^{(j)}_{n_s})$.
\end{theorem}

The functions that satisfy condition $(\ddagger)$ above include
a large class of polynomial functions with \emph{rational} coefficients.


\begin{proposition}\label{polynomials}
Let $P(z_1,\ldots,z_m)\in\Q[z_1,\ldots,z_m]$
be a polynomial in several variables over the the rational numbers
with positive leading coefficient and no constant term. Then
the polynomial function $P(z_1,\ldots,z_m)$ satisfies condition 
$(\ddagger)$ of
Theorem \ref{main5}.
\end{proposition}

\begin{proof}
Let $P(\z)=\sum_\alpha c_\alpha\z^\alpha\in\Q[z_1,\ldots,z_m]$. 
We will prove the following property:
\begin{itemize}
\item
$\exists\,N\in\N$ such that $\exists\,\overline{n}_1\ \forall n_1\ge\overline{n}_1\
\exists\,\overline{n}_2\ \forall n_2\ge\overline{n}_2\ \, \ldots\ \,
\exists\,\overline{n}_m\ \forall n_m\ge\overline{n}_m$ 
one has that $P(n_1N,n_2N,\ldots,n_mN)\in\N$.
\end{itemize}

Since all $c_\alpha\in\Q$ we can pick $N\in\N$ such that
$N\cdot c_\alpha\in\Z$ for every $\alpha\in\text{Supp}(P)$.
Then consider the polynomial 
$$P'(z_1,\ldots,z_m):=P(z_1N,\ldots,z_mN),$$
that is, $P'(\z)=\sum_\alpha c'_\alpha\z^\alpha$ where
for every multi-index $\alpha=(\alpha_1,\ldots,\alpha_n)$, it is
$c'_\alpha=N^{\alpha_1+\ldots+\alpha_n}c_\alpha$.
Since $P$ has no constant term, then it is readily verified
that $P'\in\Z[z_1,\ldots,z_m]$. Then, 
by the previous Proposition \ref{polynomials1} applied to $P'$, we obtain that:
\begin{itemize}
\item
$\exists\,\overline{n}_1\ \forall n_1\ge\overline{n}_1\
\exists\,\overline{n}_2\ \forall n_2\ge\overline{n}_2\ \, \ldots\ \,
\exists\,\overline{n}_m\ \forall n_m\ge\overline{n}_m$ 
one has that 
$P'(n_1,\ldots,n_m)>0$, and hence $P(n_1 N,n_2 N,\ldots,n_m N)\in\N$.
\end{itemize}
\end{proof}

Let us now see a few particular cases 
of Theorems \ref{main4} and \ref{main5}. 
The examples presented below are not necessarily the 
most relevant or interesting;
rather, they have been chosen with the only intent of giving
the flavor of the kind of configurations that one can obtain.

\begin{example}\ \!\!\!\!\!\footnote
{~Compare to \cite[Thm.\,1.13]{bhw}.}
Let $f:\N^3\to\Z$ be the polynomial function $f(z_1,z_2,z_3)=-3 z_1+2 z_2 z_3$.
For all $(\ell_1,k_1), (\ell_2,k_2), (\ell_3,k_3)$
where either $\ell_j>0$ divides $k_j(k_j-1)$ or $(\ell_j,k_j)=(0,1)$,
Theorem \ref{main4} applies.
E.g., let us take $(\ell_1,k_1)=(0,1)$, $(\ell_2,k_2)=(1,1)$, and $(\ell_3,k_3)=(\ell_1,k_1)=(0,1)$.
Notice that the leading term of $f$, namely $2z_2 z_3$, has positive leading coefficient
and so, by Proposition \ref{polynomials1}, condition $(\dagger)$ is satisfied.
Then we obtain the following infinite monochromatic pattern 
in the natural numbers, where the sequences $(x_n)_{n=1}^\infty$ and
$(y_n)_{n=1}^\infty$ are injective:
$$\left\{-3\sum_{s\in F_1}\!x_s+
2\left(\sum_{\emptyset\ne G\subseteq F_2}
\prod_{t\in G}y_t\right)\!\left(\sum_{u\in F_3}x_u\right)\ 
\Bigg | \ F_1<F_2<F_3\right\}.$$
In particular, if $a, b, c, d$ are the first elements of the sequence $(x_n)_{n=1}^\infty$,
$d', e, f, g$ are the first elements of the sequence $(y_n)_{n=1}^\infty$, and we only
consider those $F_1<F_2<F_3$ which are nonempty subsets of
$\{1,2,3,4\}$, then we obtain the following monochromatic pattern
in the natural numbers:\footnote
{~We can assume without loss of generality that $a,b,c,d,e,f,g$ are pairwise distinct.
The six elements of the pattern correspond to the following six choices
of $F_1<F_2<F_3$, respectively:
\begin{multline*}
\{1\}<\{2\}<\{3\},\ \{1\}<\{2\}<\{4\},\ \{2\}<\{3\}<\{4\},
\\
\{1,2\}<\{3\}<\{4\},\ \{1\}<\{2,3\}<\{4\},\ \{1\}<\{2\}<\{3,4\}.
\end{multline*}}
\begin{multline*}
-3a+2ce,\ -3a+2de,\ -3b+2df,
\\
-3a-3b+2df,\ -3a+2de+2df+2def,\ -3a+2ce+2de.
\end{multline*}
\end{example}

\begin{example}
Let $(\ell_j,k_j)=(1,1)$ for $j=1,\ldots,m$, and let $f$ be any linear function
$f(z_1,\ldots,z_m)=\sum_{j=1}^m c_j z_j$ with coefficients $c_j\in\Q$
and where $c_m>0$. Then by Theorem \ref{main5} 
we have the following infinite monochromatic pattern in
the natural numbers where the sequence $(x_n)_{n=1}^\infty$ is injective:
$$\left\{\sum_{j=1}^m c_j\left(\sum_{\emptyset\ne G\subseteq F_j}
\prod_{s\in G}x_s\right)\,\Bigg|\, F_1<\ldots<F_m\right\}.$$
\end{example}
\noindent
For instance, when $m=3$, the following elements are monochromatic:
\begin{itemize}
\item
$c_1 x_s+c_2 x_t+c_3 x_u$ 
for all $s<t<u$;
\item
$c_1 (x_s+x_t+x_s x_t)+c_2 x_u+c_3 x_v$,
$c_1 x_s+c_2 (x_t+x_u+x_t x_u)+c_3 x_v$, 
$c_1 x_s+c_2 x_t+c_3(x_u+x_v +x_u x_v)$ 
for all $s<t<u<v$; 
\item
$c_1 (x_s+x_t+x_u+x_s x_t+x_s x_u +x_t x_u+x_s x_t x_u)+c_2 x_v+c_3 x_w$,
\\
$c_1 (x_s+x_t+x_s x_t)+c_2 (x_u+x_v+x_u x_v)+c_3 x_w$,
\\
$c_1 (x_s+x_t+x_s x_t)+c_2 x_u+c_3 (x_v+x_w+x_v x_w)$,
\\
$c_1 x_s+c_2 (x_t+x_u+x_v+x_t x_u+x_t x_v+x_u x_v+x_t x_u x_v)+c_3 x_w$,
\\
$c_1 x_s+c_2 (x_t+x_u+x_t x_u)+c_3(x_v+x_w+x_u x_w)$,
\\
$c_1 x_s +c_2 x_t+ c_3(x_u+x_v+x_w+x_u x_v+x_u x_w+x_v x_w+x_u x_v x_w)$ 
\\
for all $s<t<u<v<w$; and so forth.
\end{itemize}

\begin{example}
Let $f:\N^3\to\Q$ be the function 
$$f(z_1,z_2,z_3)\ =\ -\frac{11}{5}z_1^3+\frac{1}{3}\cdot \frac{z_3}{z^2_2}.$$
Observe that $f$ satisfies condition $(\ddagger)$ of Theorem \ref{main5}
because, by letting $\overline{n}_1=1$, $N_1=5$, $\overline{n}_2=1$, $N_2=1$, 
$\overline{n}_3=275 n_1^3+1$, and $N_3=3n_2^2$, the
following property holds: 
\begin{itemize}
\item
$\exists\,\overline{n}_1,\!N_1\ \forall n_1\ge\overline{n}_1\ 
\exists\,\overline{n}_2,\!N_2\ \forall n_2\ge\overline{n}_2\ \
\exists\,\overline{n}_3,\!N_3\ \forall n_3\ge\overline{n}_3$ one has
$f(n_1N_1,n_2N_2,n_3N_3)\in\N$.
\end{itemize}
Indeed,
$$f(n_1N_1,n_2 N_2,n_3N_3)\ =\ 
-\frac{11}{5}(n_1 5)^3+
\frac{1}{3}\cdot\frac{n_3\cdot(3n_2^2)}{n_2^2}\ =\ -275n_1^3+n_3\in\N.$$

If we consider $(\ell_1,k_1)=(\ell_2,k_2)=(1,0)$ and $(\ell_3,k_3)=(1,1)$
then, by Theorem \ref{main5}, we obtain the following infinite monochromatic pattern in
the natural numbers, where the sequences $(x_n)_{n=1}^\infty$ 
and $(y_n)_{n=1}^\infty$ are injective:
$$\left\{-\frac{11}{5}\left(\prod_{s\in F_1}x_s\right)^{\!\!3}+
\frac{1}{3}\cdot\frac{\sum_{\emptyset\ne G\subseteq F_3}
\prod_{u\in G}y_u}{\left(\prod_{t\in F_2}x_t\right)^2}
\Bigg | \ F_1<F_2<F_3\right\}.$$
In particular, if $a, b, c, c'$ are the first elements of the sequence $(x_n)_{n=1}^\infty$,
$c'',c''', d, e$ are the first elements of the sequence $(y_n)_{n=1}^\infty$, and we only
consider those $F_1<F_2<F_3$ which are nonempty subsets of
$\{1,2,3,4\}$, then we obtain the following monochromatic pattern
in the natural numbers:\footnote
{~We can assume without loss of generality that $a,b,c,d,e$ are pairwise distinct.
The six elements of the pattern correspond to the following six choices
of $F_1<F_2<F_3$, respectively:
\begin{multline*}
\{1\}<\{2\}<\{3\},\ \{1\}<\{2\}<\{4\},\ \{2\}<\{3\}<\{4\},
\\
\{1,2\}<\{3\}<\{4\},\ \{1\}<\{2,3\}<\{4\},\ \{1\}<\{2\}<\{3,4\}.
\end{multline*}}
\begin{multline*}
-\frac{11}{5}a^3+\frac{1}{3}\frac{d}{b^2},\
-\frac{11}{5}a^3+\frac{1}{3}\frac{e}{b^2},\
-\frac{11}{5}b^3+\frac{1}{3}\frac{e}{c^2},\
\\
-\frac{11}{5}a^3b^3+\frac{1}{3}\frac{e}{c^2},\
-\frac{11}{5}a^3+\frac{1}{3}\frac{e}{b^2c^2},\ 
-\frac{11}{5}a^3+\frac{1}{3}\frac{d+e+de}{b^2}.
\end{multline*}
\end{example}


%
%
%
%

\begin{example}
Let $(\ell_1,k_1)=(0,1)$, $(\ell_2,k_2)=(1,1)$,
let $r\in\R\setminus\Q$ be an irrational number, 
and let $f:\N^2\to\Q$ be the function 
$$f(z_1,z_2)\ =\ \frac{\big\lfloor\{r z_1\} z_2\big\rfloor\cdot z_2}{17 z_1^3}$$
where $\lfloor\,x\,\rfloor:=\max\{s\in\Z\mid s\le x\}$ is the \emph{integer part},
and $\{x\}:=x-\lfloor\,x\,\rfloor$ is the \emph{fractional part}.
Observe that $f$ satisfies condition $(\ddagger)$ of Theorem \ref{main5}
with $\overline{n}_1=N_1=1$.
Indeed, let an arbitrary $n_1\in\N$ be given. Since $r$ is irrational,
$\{r n_1\}>0$ and so we can pick $\overline{n}_2$ such that
$\{r n_1\}\overline{n}_2\ge 1$.
By letting $N_2:=17 n_1^3$, the desired
condition $f(n_1N_1,n_2 N_2)\in\N$  is fulfilled for every $n_2\ge\overline{n}_2$.
Then we obtain the following infinite monochromatic pattern in
the natural numbers, where the sequences $(x_n)_{n=1}^\infty$ 
and $(y_n)_{n=1}^\infty$ are injective:

$$\left\{\frac
{\big\lfloor
\{r \sum_{s\in F_1}x_s\} \cdot
\sum_{\emptyset\ne G\subseteq F_2} \prod_{t\in G}y_t\big\rfloor\cdot
\sum_{\emptyset\ne G\subseteq F_2} \prod_{t\in G}y_t}
{17\cdot(\sum_{s\in F_1}x_s)^3}\
\Bigg | \ F_1<F_2\right\}.$$
In particular, if $a, b$ are the first two elements of the sequence $(x_n)_{n=1}^\infty$,
$b',c, d$ are the first three elements of the sequence $(y_n)_{n=1}^\infty$, and we only
consider those $F_1<F_2$ which are nonempty subsets of
$\{1,2,3\}$, then we obtain the following monochromatic pattern
in the natural numbers:\footnote
{~We can assume without loss of generality that $a,b,c,d$ are pairwise distinct.
The six elements of the pattern correspond to the following five choices
of $F_1<F_2$, respectively:
$$\{1\}<\{2\},\ \{1\}<\{3\},\ \{2\}<\{3\},\ \{1,2\}<\{3\},\ \{1\}<\{2,3\}.$$}

For instance, the following pattern is monochromatic
in the natural numbers, where $F_1=\{a<b\}<\{c<d\}=F_2$:
\begin{multline*}
\frac{\big\lfloor\{r a\}c\big\rfloor c}{17 a^3},\
\frac{\big\lfloor\{r a\}d\big\rfloor d}{17 a^3};\
\frac{\big\lfloor\{r b\}d\big\rfloor d}{17 b^3};\
\\
\frac{\big\lfloor\{r(a+b)\}d\big\rfloor d}{17 (a+b)^3};\
\frac{\big\lfloor\{ra\}(c+d+cd)\big\rfloor(c+d+cd)}{17 a^3}.
\end{multline*}
\end{example}

\medskip
\section{The associative operations $\lk$}\label{operationslk}

In order to prove the results presented in the previous sections,
we need to introduce a suitable class of associative operations on the integers.

For $\ell, k\in\Z$ with $\ell\ne 0$, let $T_{\ell,k}:\Z\to\Z$ be the affine transformation
$$T_{\ell,k}:a\longmapsto \ell a+k.$$ 
An elementary, but crucial observation is the following.

\begin{proposition}
Let $\ell,k\in\Z$ be integers with $\ell\ne 0$. Then the set 
$$S_{\ell,k}\ :=\ \text{range}(T_{\ell,k})\ =\ \{\ell a+k\mid a\in\Z\}$$
is closed under multiplication if and only if $\ell$ divides $k(k-1)$.
\end{proposition}

\begin{proof}
Just notice that $S_{\ell,k}=\{m\in\Z\mid m\equiv k\mod \ell\}$
is closed under multiplication if and only if $k^2\equiv k\mod \ell$.

Equivalently, given $a,b\in\Z$, there exists $c\in\Z$ such that
$(\ell a+k)(\ell b+k)=(\ell c+k)$ if and only if
$$c\ =\ \frac{1}{\ell}\left((\ell a+k)(\ell b+k)-k\right)\ =\ 
\ell ab + k(a+b) + \frac{k(k-1)}{\ell}\ \in\ \Z,$$
and this happens if and only if $\frac{k(k-1)}{\ell}\in\Z$.
\end{proof}

When $S_{\ell,k}$ is closed under multiplication,
the bijection $T_{\ell,k}:\Z\to S_{\ell,k}$ 
induces an operation $\lk$ on $\Z$ that makes $T_{\ell,k}$
an isomorphism of semigroups:
$$T_{\ell,k}:(\Z,\lk)\to(S_{\ell,k},\cdot).$$


\begin{definition}
For $\ell,k\in\Z$ where $\ell\ne 0$ divides $k(k-1)$, define:
$$a\,\lk b\,=\,c\ \Longleftrightarrow\ T_{\ell,k}(a)\cdot T_{\ell,k}(b)=T_{\ell,k}(c)\ 
\Longleftrightarrow\ (\ell a+k)(\ell b+k)\,=\,(\ell c+k).$$
\end{definition}

As seen above, the explicit formula is the following:
$$a\,\lk\,b\,=\,\ell ab+ k(a+b)+\frac{k(k-1)}{\ell}.$$
Notice that when $k=0$, one has:
$$a\,\textcircled{$\star$}_{\ell,0}\,b\,=\,\ell ab.$$

Clearly, for iterated $\lk$-products one has that
\begin{equation}\label{eq-ip}
a_1\,\lk \,\cdots\,\lk\,a_n\ =\ c\ \Longleftrightarrow\ 
(\ell a_1+k)\,\cdots\,(\ell a_n+k)\,=\,(\ell c+k).
\end{equation}

We now extend the definition of operations $\lk$ to 
the case where $\ell=0$ and $k=1$,
so as to also include the usual finite sums:
$$a\,\textcircled{$\star$}_{0,1}\,b\,=\,a+b.$$

The $(\ell,k)$-symmetric polynomials $\mathfrak{S}_{\ell,k}(X_1,\ldots,X_n)$
of Definitions \ref{def-symmetricpolynomial} and \ref{def-symmetricpolynomial2}
have been introduced because their values are precisely the iterated $\lk$-products.

\begin{proposition}\label{lk-finiteproducts}
Let $\ell,k\in\Z$ be such that 
either $\ell\ne 0$ divides $k(k-1)$ or $(\ell,k)=(0,1)$.
Then for all $a_1,\ldots,a_n\in\Z$:
$$a_1\,\lk \,\cdots\,\lk\,a_n\ =\ \mathfrak{S}_{\ell,k}(a_1,\ldots,a_n).$$
\end{proposition}

\begin{proof}
When $(\ell,k)=(0,1)$ the desired equality directly follows from the definitions.
Indeed, operation $\textcircled{$\star$}_{0,1}$
is the sum, and the $(0,1)$-symmetric polynomial 
$\mathfrak{S}_{0,1}(X_1,\ldots,X_n)=X_1+\ldots+X_n$.
Also when $k=0$ and $\ell\ne 0$, one directly uses the definitions, since 
$$a\,\textcircled{$\star$}_{\ell,0}\,b\,=\,\ell ab\ \Longrightarrow\ 
a_1\,\textcircled{$\star$}_{\ell,0}\,\cdots\,\textcircled{$\star$}_{\ell,0}\,a_n\,=\,\ell^n a_1\cdots a_n\ =\ 
\mathfrak{S}_{\ell,0}(a_1,\ldots,a_n).$$

Finally, let us now assume that $\ell,k\ne 0$ and $\ell$ divides $k(k-1)$. 
By equality $(\ref{eq-ip})$, if $c=a_1\,\lk \,\cdots\,\lk\,a_n$ then 
\begin{multline*}
c\ =\ \frac{1}{\ell}\left(\prod_{j=1}^n(\ell a_j+k)\right)-\frac{k}{\ell}\ =\ 
\frac{1}{\ell}\left(k^n+\sum_{\emptyset\ne G\subseteq\{1,\ldots,n\}}
\ell^{|G|} k^{n-|G|}\prod_{s\in G}a_s\right)-\frac{k}{\ell}\ =
\\
=\ \sum_{\emptyset\ne G\subseteq\{1,\ldots,n\}}
\left(\ell^{|G|-1}k^{n-|G|}\cdot\prod_{s\in G}a_s\right)\ +\ \frac{k^n-k}{\ell}\ =\ 
\\
=\ \sum_{j=1}^{n} \ell^{j-1} k^{n-j} e_{j}(a_1,\ldots,a_n) + \frac{k^n-k}{\ell}\ =\ 
\mathfrak{S}_{\ell,k}(a_1,\ldots,a_n).
\end{multline*}
\end{proof}

Before showing the fundamental 
properties that are satisfied by the operations $\lk$, we review the
basic notions about semigroups (see, e.g., the monograph \cite{ho}).

Recall that a \emph{semigroup} $(S,\star)$ is a set $S$ endowed
with an associative operation $\star$.
An element $z$ is a \emph{zero} if $z\star a=a\star z=z$ for every $a\in A$;
and an element $u$ is an \emph{identity} if $u\star a=a\star u=a$ for every $a\in S$.
If a zero element or an identity element exist, 
then they are necessarily unique.
An element $a$ is \emph{invertible} if it has an inverse $b$, that is
$a\star b=b\star a=u$.

For simplicity, in the following we will write $a^{(n)}$ 
to denote the $n$-th power of $a$ with respect to the operation $\star$:
$$a^{(n)}\ :=\ \underbrace{a\star\cdots\star a}_{n\ \text{times}}$$

An element $a$ has \emph{finite order} (or infinite order)
if the generated sub-semigroup 
$\{a^{(n)}\mid n\in\N\}$ is finite (or infinite, respectively);
equivalently, $a$ has finite order if $a^{(n)}=a^{(m)}$ for some $n\ne m$.

The semigroup $(S,\star)$ is \emph{left cancellable} if 
every element $a$ is left cancellable, that is for all $b,b'$,
one has that $a\star b=a\star b'\Rightarrow b=b'$. The notion of
\emph{right cancellable} is defined similarly.
A semigroup is \emph{cancellative} if it is both left and right cancellable.
Clearly, for commutative semigroups, the notions
of left cancellativity, right cancellativity, and cancellativity coincide.

\begin{proposition}\label{lk}
Let $\ell,k\in\Z$ be such that $\ell\ne 0$ divides $k(k-1)$. Then 
\begin{enumerate}
\item
$(\Z,\lk)$ is a commutative semigroup.
\item
$(\Z,\lk)$ contains the zero element $z$ if and only if
$\ell$ divides $k$; in this case, $z=-\frac{k}{\ell}$.
\item
$(\Z,\lk)$ contains the identity element $u$ 
if and only if $\ell$ divides $k-1$; in this case,
$u=-\frac{k-1}{\ell}$.
\item
$(\Z,\lk)$ contains an invertible element $u'\ne u$ if
and only $\ell$ divides $k+1$; in this case the only such element is
$u'=-\frac{k+1}{\ell}$ and $u'\lk u'=u$.
Therefore, $(\Z,\lk)$ is not a group.
\item
The only possible elements of $(\Z,\lk)$ that have finite order are the zero element 
$z=-\frac{k}{\ell}$ of order 1, the identity $u=-\frac{k-1}{\ell}$
of order 1, and $u'=-\frac{k+1}{\ell}$ of order 2, when they are integers.
\item
$\Z\setminus\{-\frac{k}{\ell}\}$ is a cancellative sub-semigroup.
\item
$\Z\setminus\{-\frac{k+1}{\ell},-\frac{k}{\ell}\}$ is a cancellative sub-semigroup.
\item
$\Z\setminus\{-\frac{k+1}{\ell},-\frac{k}{\ell},-\frac{k-1}{\ell}\}$
is a cancellative sub-semigroup.
\item
If $\ell>0$ then $\{a\in\Z\mid a>-\frac{k}{\ell}\}$ is a cancellative sub-semigroup.
\item
If $\ell>0$ and $N\in\N$ then 
$\{a\in\Z\mid a\ge -\frac{k}{\ell}+\frac{N}{\ell}\}$ is a cancellative sub-semigroup.
\end{enumerate}
\end{proposition}

Notice that under the hypothesis that $\ell\ne 0$ divides $k(k-1)$,
if $\ell$ divides $k+1$ then $\ell$ also divides $k-1=(k+1)(k-1)-k(k-1)$, and
so $(\Z,\lk)$ contains the identity $u$.
The converse does not hold; \emph{e.g.} if $\ell=3$ and $k=4$
then $\ell$ divides $k-1$, and hence $k(k-1)$, but $\ell$ does not divide $k+1$;
in this case, the identity $u=-\frac{k-1}{\ell}=-1$ is the only invertible element in
$(\Z,\textcircled{$\star$}_{3,4})$
\begin{proof}
All properties directly follow from the fact that,
when $\ell\ne 0$ divides $k(k-1)$, the affine transformation
$T_{\ell,k}:(\Z,\lk)\to (S_{\ell,k},\,\cdot\,)$ is an isomorphism of semigroups.

(1). The associativity and commutativity properties of multiplication on $S_{\ell,k}$
are inherited by the operation $\lk$, via the isomorphism $T_{\ell,k}$.

%

(2). $z$ is the zero element of $(\Z,\lk)$ if and only if $T_{\ell,k}(z)=\ell z+k=0$
is the zero element of $S_{\ell,k}$ if and only if $z=-\frac{k}{\ell}\in\Z$.


(3). $u$ is the identity of $(\Z,\lk)$ if and only if $T_{\ell,k}(u)=\ell u+k=1$
is the identity of $S_{\ell,k}$ if and only if $u=-\frac{k-1}{\ell}\in\Z$.


(4). A product $a\,\lk b=u$ if and only if $(\ell a+k)(\ell b+k)=1$
if and only if either $(\ell a+k)=(\ell b+k)=1$ or
$(\ell a+k)=(\ell b+k)=-1$.
In the former case $a=b=u=-\frac{k-1}{\ell}$, and in the latter case 
$a=b=u'=-\frac{k+1}{\ell}$ and so $u'\,\lk u'=u$.


(5). An element $a\ne z=-\frac{\ell}{k}$ has finite order if and only if
$\ell a+k\ne 0$ has finite order in $(S_{\ell,k},\,\cdot\,)$.
But then it must be either $\ell a+k=1$ and hence $a=u$,
or $\ell a+k =-1$ and hence $a=u'=-\frac{k+1}{\ell}$.

(6). If $a,b\ne-\frac{k}{\ell}$ 
and $c=a\,\lk b$ then $(\ell c+k)=(\ell a+k)(\ell b+k)\ne 0$ 
and so $c\ne-\frac{k}{\ell}$. This proves that $\Z\setminus\{-\frac{k}{\ell}\}$ is a subsemigroup.
Let us now show that every $b\ne-\frac{k}{\ell}$ is cancellative.
By definition, $a\,\lk b=a'\,\lk b$ if and only if $(\ell a+k)(\ell b+k)=(\ell a'+k)(\ell b+k)$.
Since $b\ne-\frac{k}{\ell}$, we can conclude that
$\ell a+k=\ell a'+k$, and hence $a=a'$.

(7) and (8). Notice that $\Z\setminus\{-1,0\}$ and $\Z\setminus\{-1,0,1\}$
are multiplicative sub-semigroups of $\Z$, and hence
$S_{\ell,k}\setminus\{-1,0\}$ and $S_{\ell,k}\setminus\{-1,0,1\}$
are sub-semigroups of $(S_{\ell,k},\,\cdot\,)$.
Since $T_{\ell,k}(u')=-1$, $T_{\ell,k}(z)=0$, and $T_{\ell,k}(u)=1$,
it follows that $\Z\setminus\{u',z\}$ and $\Z\setminus\{u',z,u\}$
are sub-semigroups of $(\Z,\lk)$. 
%
%

(9). Since $\ell>0$, one has that $x>-\frac{k}{\ell}\Leftrightarrow (\ell x+k)>0$.
If $a,b>-\frac{k}{\ell}$ and $c=a\,\lk b$ then $(\ell c+k)=(\ell a+k)(\ell b+k)>0$ 
and so also $c>-\frac{k}{\ell}$. 

(10). Similarly as in the previous point,
since $\ell>0$ one has that $x\ge-\frac{k}{\ell}+\frac{N}{\ell}\Leftrightarrow (\ell x+k)\ge N$.
If $a,b\ge-\frac{k}{\ell}+\frac{N}{\ell}$ and $c=a\,\lk b$ then $(\ell c+k)=(\ell a+k)(\ell b+k)\ge N^2\ge N$,
and hence also $c\ge-\frac{k}{\ell}+\frac{N}{\ell}$. 

Finally, notice that the semigroups considered in (7), (8), (9), and (10) are
cancellative as sub-semigroups of $\Z\setminus\{-\frac{k}{\ell}\}$.
\end{proof}

%
%
%

\medskip
\section{The proofs}

In this section we briefly review several general properties of 
algebra in the Stone-\v{C}ech compactification of semigroups, 
and then apply them to the semigroups determined by associative operations $\lk$.
The reader is referred to the fundamental 
book \cite{hs} for a comprehensive presentation of all notions
and results recalled here.

\subsection{Algebra in the Stone-\v{C}ech compactification}

The primary observation on which the theory is grounded is the
fact that any associative operation $\star$ on a discrete set $S$
can be extended to an associative operation on its
\emph{Stone-\v{C}ech compactification}
$\beta S=\{\U\mid \U\ \text{ultrafilter on}\ S\}$.
Recall that a base of open and closed sets of the topology on $\beta S$ 
is given by the family $\{\mathcal{O}_A\mid A\subseteq S\}$
where $\mathcal{O}_A=\{\U\in\beta S\mid A\in\U\}$.
It is assumed that $S\subseteq\beta S$ by identifying each element $a\in S$ 
with the principal ultrafilter $\U_a=\{A\subseteq S\mid a\in A\}$.

For all $\U,\V\in\beta S$, the ultrafilter $\U\star\V$
is defined by letting for every $A\subseteq S$:
$$A\in \U\star\V\ \Longleftrightarrow\ 
\left\{a\in S\mid a^{-1}A\in\V\right\}\in\U$$
where $a^{-1}A:=\{b\in S\mid a\star b\in S\}$.
Notice that the above is an actual extension of the operation on $S$,
since $\U_a\star\U_b=\U_{a\star b}$.  
The resulting semigroup $(\beta S,\star)$ has the structure of a 
\emph{compact right topological semigroup} (see \cite[\S 4.1]{hs}),
that is, for every $\V\in\beta S$ the ``product on the right" 
$\rho_\V:\U\mapsto \U\star\V$ is a continuous function. 

The most considered examples are $(\beta\N,\oplus)$ and $(\beta\N,\odot)$,
namely the semigroups obtained on the Stone-\v{C}ech compactification
of the natural numbers from the additive semigroup $(\N,+)$ and the multiplicative semigroup
$(\N,\cdot)$, respectively. In fact, the study of those ultrafilter semigroups
have produced a remarkable amount of results in arithmetic Ramsey Theory, 
as evidenced by the extensive monograph \cite{hs}.
It is worth noticing that in virtually all significant examples, including
$S=(\N,+)$ and $S=(\N,\cdot)$, the 
ultrafilter semigroup $(\beta S,\star)$ is not commutative
(see \cite[\S 4.2]{hs}).

A fundamental tool in this area of research is provided by 

\begin{itemize}
\item
Ellis' Lemma \cite{el}:
\emph{In every compact right topological semigroup there exist
idempotent elements $x\star x=x$.}\footnote
{~This is Theorem 2.5 of \cite{hs}.}
\end{itemize}
 
In consequence, for every semigroup $(S,\star)$ there exist
idempotent ultrafilters $\U=\U\star\U$ in $(\beta S,\star)$.

%
%


For any sequence $(a_n)_{n=1}^\infty$ of elements in a semigroup $(S,\star)$, 
denote by $\FP(a_n)_{n=1}^\infty$ the corresponding set of \emph{finite products}:
$$\FP(a_n)_{n=1}^\infty\ :=\ \left\{a_{n_1}\star\,\cdots\,\star a_{n_s}\mid n_1<\ldots<n_s\right\}.$$ 

The relevance of idempotent ultrafilters in Ramsey Theory
is based on the following crucial fact:

\begin{itemize}
\item
Galvin's Theorem:\
\emph{Let $(S,\star)$ be a semigroup, and let
$\U=\U\star\U$ be an idempotent ultrafilter in the Stone-\v{Cech} compactification
$(\beta S,\star)$. Then for every $A\in\U$ 
there exists a sequence $(a_n)_{n=1}^\infty$
such that the set of finite products $\FP(a_n)_{n=1}^\infty\subseteq A$.
}
\end{itemize}


To our purposes, we need the following general property
about \emph{non-principal} ultrafilters and \emph{injective} sequences.

\begin{theorem}\label{nonprincipalidempotents}
Let $(S,\star)$ be a semigroup, and assume that
there is a left cancellable subsemigroup $(S',\star)$ where
$S\setminus S'$ is finite. 
Then a set $A\subseteq S$
includes a set of finite products $\FP(a_n)_{n=1}^\infty\subseteq A$
for some injective sequence $(a_n)_{n=1}^\infty$ if and only if $A\in\U$
for some non-principal idempotent ultrafilter $\U=\U\star\U$ on $S$.
\end{theorem}

\begin{proof}
This is just a variant of \cite[Thm. 5.12]{hs}, where one considers
\emph{non-principal} ultrafilters and \emph{injective} sequences.
However, for completeness, we outline the proof here.

Notice first that our hypothesis on $S$ guarantees that
the non-principal ultrafilters $\beta S\setminus S$ form a sub-semigroup. 
Indeed let $\U,\V\in\beta S\setminus S$ and let $c\in S$.
By left cancellativity, for every $a\in S'$ the set $\{b\in S'\mid a\star b=c\}$
contains at most one element, and so $\{b\in S\mid a\star b=c\}\notin\V$, 
because it is finite. Then $\{a\in S\mid\{b\in S\mid a\star b=c\}\notin\V\}\in\U$, since
it includes $S'$, which is a cofinite subset of $S$. This means that
$\{c\}\notin\U\star\V$. As $c\in S$ was arbitrary, 
we can conclude that $\U\star\V$ is non-principal.
Now recall that given any sequence $(x_n)_{n=1}^\infty$, the intersection
$X:=\bigcap\{\mathcal{O}_{\FP(x_n)_{n=s}^\infty}\mid s\in\N\}$
is a closed sub-semigroup of $(\beta S,\star)$ (see \cite[Lemma 5.11]{hs}).
Then also $(X\cap(\beta S\setminus S),\star)$ is a closed sub-semigroup
of $(\beta S,\star)$. Notice that, since $(x_n)_{n=1}^\infty$ is injective,
the family of sets $\{\FP(x_n)_{n=s}^\infty\mid s\in\N\}\cup\{\N\setminus\{s\}\mid s\in\N\}$
has the finite intersection property and so, by compactness of $\beta S$,
its intersection is nonempty:
$$(\beta S\setminus S)\cap X\ =\ 
\bigcap_{s\in\N}\mathcal{O}_{\FP(x_n)_{n=s}^\infty}\,\cap\,
\bigcap_{s\in\N}\mathcal{O}_{\N\setminus\{s\}}\ \ne\ \emptyset.$$
Then, by Ellis' Lemma, there exist idempotent elements $\U\in(\beta S\setminus S)\cap X$.
Such non-principal ultrafilters $\U$ contains $\FP(x_n)_{n=1}^\infty\in\U$,
and hence $A\in\U$.

For the other direction, it is readily seen that
the sequence $(x_n)_{n=1}^\infty$ as constructed in Glazer's Theorem
(see \cite[Thm. 5.8]{hs}) 
can be made injective by assuming that the idempotent ultrafilter $\U$ is non-principal.
\end{proof}

%
%
%
%
%

\subsection{Proof of Theorem \ref{main1}}

This is similar to the ultrafilter proof of Hindman Theorem
where the associative operation $\lk$ is considered
instead of the sum operation, the only difference being
that one has to consider suitable subsets of the integers
so as to have cancellative sub-semigroups of $(\Z,\lk)$.
 
Let $(x_n)_{n=1}^\infty$ be an injective sequence of integers, and let
$\mathfrak{S}_{\ell,k}(x_n)_{n=1}^\infty=C_1\cup\ldots\cup C_r$ be a finite coloring
of the corresponding $(\ell,k)$-symmetric system.
By the hypothesis that $\ell\ne 0$ divides $k(k-1)$, we have that
$(\Z,\lk)$ is a semigroup; besides, by Proposition \ref{lk-finiteproducts}, the set
of finite $\lk$-products $\FP(x_n)_{n=1}^\infty$ coincides with the $(\ell,k)$-symmetric system
$\mathfrak{S}_{\ell,k}(x_n)_{n=1}^\infty$.
Since $(\Z,\lk)$ has $\Z\setminus\{-\frac{\ell}{k}\}$ as a cancellative sub-semigroup,
by Theorem \ref{nonprincipalidempotents} we can pick a non-principal
idempotent ultrafilter $\U=\U\lk\U$ on $\Z$ such that $\mathfrak{S}_{\ell,k}(x_n)_{n=1}^\infty\in\U$.
Then one color of the partition $C_i\in\U$ and so, again by Theorem \ref{nonprincipalidempotents},
$\mathfrak{S}_{\ell,k}(y_n)_{n=1}^\infty\subseteq C_i$ for a suitable injective
sequence $(y_n)_{n=1}^\infty$.

\smallskip
Now let us assume that $\ell\in\N$ is positive, and let
$\N=C_1\cup\ldots\cup C_r$ be a finite coloring.
Pick a natural number $N\in\N$ with $N>-\frac{k}{\ell}$.
By Proposition \ref{lk}, if we consider 
the subset $\N':=\{a\in\Z\mid a>-\frac{k}{\ell}+\frac{N}{\ell}\}\subseteq\N$
then $(\N',\lk)$ is a cancellative semigroup, and so
we can pick a non-principal idempotent ultrafilter $\U=\U\lk\U$ on $\N'$.
By the property of ultrafilter, 
there exists $i$ such that $C_i\cap\N'\in\U$, and so, by
Theorem \ref{nonprincipalidempotents}, there exists 
an injective sequence $(x_n)_{n=1}^\infty$ with
$\mathfrak{S}_{\ell,k}(x_n)_{n=1}^\infty\subseteq C_i\cap\N'\subseteq C_i$, as desired.

\subsection{Proof of Theorem \ref{main2}}

We will use a generalization of
Deuber Theorem for commutative semirings that has been recently proved 
by V. Bergelson, J.H. Johnson, and J. Moreira.
In particular, we will use the following result
(see \cite[Corollary 3.7]{bjm}):

\begin{theorem}[Bergelson-Johnson-Moreira]\label{bjm}
Let $(S,\star)$ be a commutative semigroup, and for $j=1,\ldots,m$
let $\mathfrak{F}_j$ be a finite set of endomorphisms $f:S^j\to S$.\footnote
{~The Cartesian product $S^j$ has the natural structure of semigroup
inherited from $(S,\star)$. Precisely, the associative operation 
$\star^j$ on $S^j$ is defined  coordinatewise:
$$(s_1,\ldots,s_j)\star^j(s'_1,\ldots,s'_j)\ :=\ (s_1\star s'_1,\ldots,s_j\star s'_j).$$}
Then for every finite coloring $S=C_1\cup\ldots\cup C_r$
there exist a color $C_i$ and elements $a_0,a_1,\ldots,a_m\ne u$ different
from the identity, such that $a_0\in C_i$ and $f(a_0,\ldots,a_{j-1})\star a_j\in C_i$
for every $j=1,\ldots,m$ and for every $f\in\mathfrak{F}_j$.
\end{theorem}

Notice that when $c=1$ the above property actually generalizes Deuber Theorem. 
Indeed, given $m$ and $p$, let $(S,\star)=(\N,+)$, and for $j=1,\ldots,m$ let 
$$\mathfrak{F}_j\ =\ \left\{f_{(n_0,\ldots,n_{j-1})}\,\Big|\, n_0,\ldots,n_{j-1}\in\{-p,\ldots,p\}\right\}$$
where $f_{(n_0,\ldots,n_{j-1})}:\N^j\to\N$ is the homomorphism
of semigroups given by the linear combinations
$f_{(n_0,\ldots,n_{j-1})}(a_0,\ldots,a_{j-1})=\sum_{s=0}^{j-1}n_s a_s$.
Then, for every finite coloring of $\N$,  one obtains the existence of elements 
$a_0,\ldots,a_m\in\N$
such that $a_0$ and $a_j+\sum_{s=0}^{j-1}n_s a_s$ are monochromatic
for every $j=1,\ldots,m$ and for all $n_0,\ldots,n_{j-1}\in\{-p,\ldots,p\}$.

%

\smallskip
Let $\Z':=\Z\setminus\{-\frac{k+1}{\ell},-\frac{k}{\ell}\}$.
By the properties in Proposition \ref{lk}, 
$(\Z',\lk)$
is a cancellative commutative semigroup without zero element
and, since $\ell$ divides $k-1$, with identity $u=-\frac{k-1}{\ell}\in\Z'$.

Recall that for $a\in\Z$ and $n\in\N$, 
we denoted by $a^{(n)}$ (not to be confused with $a^n$)
the $n$-th power of $a$ with respect to the operation $\lk$:
$$a^{(n)}\ :=\ \underbrace{a\,\lk\,\cdots\,\lk a}_{n\ \text{times}}$$
Extend the above notation to $n=0$ by setting $a^{(0)}=u$, the identity element.
For every $j$-tuple $\nu=(\nu_0,\ldots,\nu_{j-1})\in(\N\cup\{0\})^j$ of non-negative integers,
let $\varphi_\nu$ be the function where
$$\varphi_\nu:(a_0,\ldots,a_{j-1})\,\longmapsto\, 
a_0^{(\nu_0)}\,\lk\,\cdots\,\lk\, a_{j-1}^{(\nu_{j-1})}.$$
Since $\Z'$ is commutative, $\varphi_\nu:(\Z')^j\to\Z'$ is a semigroup homomorphism.
Now apply the above Theorem \ref{bjm} with 
$(S,\star)=(\Z',\lk)$, and with the following sets of homomorphisms for $j=1,\ldots,m$:
$$\mathfrak{F}_j\ =\ \left\{\varphi_\nu:(\Z')^j\to\Z'\,\big|\,
\nu=(\nu_0,\ldots,\nu_{j-1})\in\{0,1,\ldots,L\}^j\right\}$$

Then for every finite coloring $\Z=C_1\cup\ldots\cup C_r$ there
exist a color $C_i$ and elements $a_0,a_1,\ldots,a_m\ne u$ different
from the identity such that:
\begin{itemize}
\item
$a_0\in C_i\cap\Z'$;
\item
$a_0^{(\nu_0)}\lk\, \cdots\,\lk\, a_{j-1}^{(\nu_{j-1})}\lk\, a_j\in C_i\cap\Z'$ for every $j=1,\ldots,m$
and for all $\nu_0,\ldots,\nu_{j-1}\in\{0,1,\ldots,L\}$.
\end{itemize}

%

Finally, notice that for every $j=1,\ldots,m$ and for all non-negative
integers $\nu_0,\ldots,\nu_{j-1}$, one has
\begin{multline*}
c\ =\ a_0^{(\nu_0)}\lk\, \cdots\,\lk\, a_{j-1}^{(\nu_{j-1})}\lk\, a_j\ \Longleftrightarrow\
\\
(\ell c+k)=(\ell a_0+k)^{\nu_0}\cdots(\ell a_{j-1}+k)^{\nu_{j-1}}(\ell a_j+k)\ \Longleftrightarrow
\\
c\ = \frac{1}{\ell}\left((\ell a_j+k)\cdot\prod_{s=0}^{j-1}(\ell a_s+k)^{\nu_s}-k\right).
\end{multline*}
For every $j=0,1,\ldots,m$, since $a_j\ne u=-\frac{k-1}{\ell}$ and since 
$a_j\notin\{-\frac{k+1}{\ell},-\frac{k}{\ell}\}$,
we have that $\ell a_j+k\ne 0,1,-1$, as desired.

\smallskip
Now assume that $\ell\in\N$ is positive, and let
$\N=C_1\cup\ldots\cup C_r$ be a finite coloring.
By Proposition \ref{lk}, $(\N',\lk)$ where
$\N':=\{a\in\Z\mid a>-\frac{k}{\ell}\}$ is 
a cancellative commutative semigroup without zero element
and with identity $u=-\frac{k-1}{\ell}$.
If $k<0$, then $\N'\subseteq\N$, and we can proceed axactly as above.
If $k>0$ then consider the finite coloring $\N'=C_1\cup\ldots\cup C_r\cup F$,
where $F=\{0,-1,\ldots,-\lfloor\frac{k}{\ell}\rfloor\}$ is finite.
Notice that for large enough $m$ and $L$, the monochromatic
configuration cannot be included in the finite set $F$, since
elements $\ell a_j+k\ne 0,1,-1$.
Then proceeding as done above one finds a monochromatic
configuration in one of the $C_i$.

%

\subsection{Proof of Theorem \ref{main3}}

Let $\Z':=\Z\setminus\{-\frac{k+1}{\ell},-\frac{k}{\ell}\}$.
By Proposition \ref{lk},
$(\Z',\lk)$ is a commutative cancellative semigroup with identity $u=-\frac{k-1}{\ell}$,
where $u$ is the only invertible element,
and where all elements except $u$ have infinite order.
For $j=0,1,\ldots,L+1$ let $f_j:\Z'\to\Z'$ be the endomorphism
where $f_j(b)=b^{(j)}$.
By Theorem \ref{bjm}, for every finite coloring $\Z=C_1\cup\ldots\cup C_r$ there
exist a color $C_i$ and elements $b,a'\ne u$ different
from the identity such that such that
$$b\,,\ a'\,,\ a'\,\lk b\,,\ 
a'\,\lk b\,\lk b\,,\ 
\ldots\ ,\ a'\,\lk b^{(L+1)}\ \in\ C_i\cap\Z'.$$ 
If we let $a:=a'\,\lk b$, we have the following monochromatic pattern
$$b\,,\ a\,,\ a\,\lk b\,,\ 
\ldots\ ,\ a\,\lk b^{(L)}\ \in\ C_i\cap\Z'$$ 
where elements are pairwise distinct. 
Indeed, by cancellativity, $a'\ne u$ implies that $a:=a'\,\lk b\ne u\,\lk b=b$.
If it was $a=a\,\lk b^{(s)}$ for some $s\ge 1$ then, by cancellativity,
we would have $b^{(s)}=u$, a contradiction because $b\ne u$ has infinite order;
and if it was $b=a\,\lk b^{(s)}$ for some $s\ge 1$ then, 
again by cancellativity, we would have $a\,\lk b^{(s-1)}=u$, and hence 
$a=b^{(s-1)}=u$, a contradiction.
Finally, notice that for every $j\in\N$:
\begin{multline*}
c\ =\ a\,\lk b^{(j)}\ \Longleftrightarrow\
(\ell c+k)=(\ell a+k)(\ell b+k)^j
\\
\Longleftrightarrow\ 
c\ = \frac{1}{\ell}\big((\ell a+k)(\ell b+k)^j-k\big).
\end{multline*}

%
%
%
%
%
%
%
%
%
%

Now assume that $\ell\in\N$. 
Similarly as done in the proof
of the previous Theorem \ref{main2}, consider 
the cancellative semigroup $\N':=\{a\in\Z\mid a>-\frac{k}{\ell}\}$
with identity $u=-\frac{k-1}{\ell}$, and where all elements except $u$ have infinite order.
If $k<0$ then $\N'\subseteq\N$ and we can proceed as in the fisrt part of this proof.
If $k>0$, consider the finite coloring $\N=C_1\cup\ldots\cup C_r\cup C_{r+1}$
where $C_{r+1}:=\{0,-1,\ldots,-\lfloor\frac{k}{\ell}\rfloor\}$.
Without loss of generality we can assume that $L>|C_{r+1}|$. 
By using Theorem \ref{bjm},
as already done in the first part of the proof, we see that
there exist a color $C_i$ and elements $b,a\ne u$ such that
$$b\,,\ a\,,\ a\,\lk b\,,\ 
a'\,\lk b\,\lk b\,,\ 
\ldots\ ,\ a,\lk b^{(L')}\ \in\ C_i$$ 
where the above are pairwise different.
Clearly $i\ne r+1$ because we assumed $|C_{r+1}|> L$, and so
we obtain the desired result.

\subsection{Generalizations of Milliken-Taylor Theorem}

Several different generalizations of Milliken-Taylor Theorem to
arbitrary semigroups have been demonstrated in recent years
(see \cite[\S 3]{bhw} and \cite[Thm. 6.3]{lu}). 
Before stating the result that we need,
let us recall a few more notions about ultrafilters.

If $\U$ and $\V$ are ultrafilters on a set $I$,
the \emph{tensor product} $\U\otimes\V$ is the ultrafilter
on $I\times I$ defined by setting
$$X\in\U\otimes\V\ \Longleftrightarrow\ 
\left\{i\in I\mid\{j\in I\mid (i,j)\in X\}\in\V\right\}\in\U.$$
If one identifies -- as it is usually done --
the Cartesian products $(I\times I)\times I$ and $I\times(I\times I)$,
then it is shown that $\otimes$ is an associative
operation, that is $(\U\otimes\V)\otimes\W=\U\otimes(\V\otimes\W)$.
In consequence, one can consider iterated tensor products
$\U_1\otimes\cdots\otimes\U_m$ with no ambiguity.

We will use the following characterization of sets in tensor
products of idempotent ultrafilters, which directly implies
a general version of Milliken-Taylor Theorem.
It is a special case of the more general \cite[Corollary 3.5]{bhw},
with the variant that cancellativity is also assumed so as to obtain 
\emph{non-principal} idempotent ultrafilters and \emph{injective} sequences.

\begin{theorem}[Bergelson-Hindman-Williams]\label{mt}
For $j=1,\ldots,m$, let $\ostar_j$ be an associative and cancellative
operation on the set $S$. Then for every set $B\subseteq S^m$ the following
properties are equivalent:
\begin{enumerate}
\item
$B\in \U_1\otimes\ldots\otimes\U_m$, where $\U_j=\U_j\ostar_j\U_j$ is a
non-principal idempotent ultrafilter of $(\beta S,\ostar_j)$  for every $j$,  
and where $\U_j=\U_{j'}$ whenever $\ostar_j=\ostar_{j'}$;
\item
For $j=1,\ldots,m$ there exist injective sequences $(x_{j,n})_{n=1}^\infty$ 
with $x_{j,n}=x_{j',n}$ whenever $\ostar_j=\ostar_{j'}$, and such that
$$\left\{\left(x^{(1)}_{F_1},\ldots,x^{(m)}_{F_m}\right)\,\Big|\,F_1<\ldots<F_m\right\}\,\subseteq\,B$$
where for finite $F=\{n_1<\ldots<n_s\}$ and $1\le j\le m$ we denoted
$x^{(j)}_{F}:=x_{j,n_1}\ostar_j\cdots\ostar_j x_{j,n_s}$.
\end{enumerate}
\end{theorem}

Recall that if $\W$ is an ultrafilter on a set $X$ and $f:X\to Y$ is any function,
then the \emph{image ultrafilter} $f(\W)$ on $Y$ is defined by setting
$$f(\W)\ :=\ \{B\subseteq Y\mid f^{-1}(B)\in\W\}.$$
Notice that $A\in\W\Rightarrow f(A)\in f(\W)$, but not conversely.

If $\star$ is an associative operation on $S$, and we denote
by $f_\star:S\times S\to S$ the corresponding
binary function $f_\star:(a,b)\longmapsto a\star b$, then it is readily verified
that the extension of $\star$ to the Stone-\v{C}ech compactification $\beta S$
is given by ultrafilter images of tensor products. Precisely, for all $\U,\V\in\beta S$:
$$\U\star\V\ =\ f_\star(\U\otimes\V).$$
For instance, if $f:\N\times\N\to\N$ is the sum function
$f(n,m)=n+m$, then the image ultrafilter $f(\U\otimes\V)=\U\oplus\V$
is the usual extension of the sum in the Stone-\v{Cech} compactification $\beta\N$.

We will use the following straightforward consequence of the previous theorem,
that seems never to have been formulated explicitly.

\begin{corollary}\label{mt-corollary}
Let $f:S^m\to T$ be any function, and let $A\subseteq T$.
Then the equivalence given by the previous theorem also holds in the modified 
formulation where in (1) we consider 
the condition $A\in f(\U_1\otimes\ldots\otimes\U_m)$, 
and where in (2) we consider the condition
$$\left\{f\left(x^{(1)}_{F_1},\ldots,x^{(m)}_{F_m}\right)\,\Big|\,F_1<\ldots<F_m\right\}\,\subseteq\,A$$
\end{corollary}

\begin{proof}
Given $A\in f(\U_1\otimes\cdots\otimes\U_m)$, apply Theorem \ref{mt} to the preimage
$B:=f^{-1}(A)\in\U_1\otimes\cdots\otimes\U_m$.
\end{proof}

If $\U$ is an ultrafilter on $\N$ and $a\in\N$, then
the ultrafilter $a\U$ is defined by setting:
$$A\in a\U\ \Longleftrightarrow\ A/a:=\{n\in\N\mid n a\in A\}\in\U.$$
We remark that $a\U$ is not the same as $\mathfrak{U}_a\odot\U$,
where $\mathfrak{U}_a$ is the principal ultrafilter generated by $a$.

Now consider the semigroup $(S,\star)=(\N,+)$.
If we take as $f:\N^m\to\N$ the linear function 
$f(x_1,\ldots,x_m)=a_1x_1+\ldots+a_m x_m$ where $a_i\in\N$,
then the corollary above yields an arithmetic formulation of Milliken-Taylor Theorem 
that is well-known, namely:

\begin{itemize}
\item
``Arithmetic" Milliken-Taylor Theorem. 
\emph{Let $a_1,\ldots,a_m\in\N$.
Then the following properties are equivalent for every $A\subseteq\N$:
\begin{enumerate}
\item
$A\in a_1\U\oplus\cdots\oplus a_m\U$ for a suitable
non-principal idempotent ultrafilter $\U$ on $\N$.
\item
There exists an injective sequence $(x_n)_{n=1}^\infty$ such that
$$\left\{a_1\cdot x_{F_1}+\ldots+a_n\cdot x_{F_m}\mid F_1<\ldots<F_m\right\}\subseteq A$$
where for $F=\{n_1<\ldots<n_s\}$ we denoted $x_F:=x_{n_1}+\ldots+x_{n_s}$.
\end{enumerate}}
\end{itemize}

Indeed, it is easily verified that for every
ultrafilter $\U$ on $\N$ one has that the image ultrafilter
$f(\U\otimes\cdots\otimes\U)=a_1\U\oplus\ldots\oplus a_m\U$.\footnote
{~$f(z_1,\ldots,z_m)=a_1 z_1+\ldots+a_m z_m$ 
is just one simple example of a function that is ``coherent" with respect to tensor products.
Indeed, for every polynomial $P(z_1,\ldots,z_m)$ over $\N$ there is a canonical polynomial
function $\widetilde{P}:(\beta\N)^m\to\beta\N$ such that the image ultrafilter
$P(\U_1\otimes\cdots\otimes\U_m)=\widetilde{P}(\U_1,\ldots,\U_m)$
for all ultrafilters $\U_1,\ldots,\U_m$. The definition of $\widetilde{P}$ is obtained from the definition 
of $P$ by replacing the sum $+$ and the product $\cdot$ on $\N$ with their canonical 
extensions $\oplus$ and $\odot$ on $\beta\N$, respectively (see \cite{wi}).
More generally, in \cite[\S 3]{bhw} the class of \emph{extended polynomials} 
$f(z_1,\ldots,z_m)$ is introduced for any given set of associative operations on a set $S$ 
in such a way that one can naturally define the corresponding functions $\widetilde{f}$
that satisfy $f(\U_1\otimes\cdots\otimes\U_m)=\widetilde{f}(\U_1,\ldots,\U_m)$
for all ultrafilters $\U_1,\ldots,\U_m$ on $S$.}

%

\subsection{Proofs of Theorems \ref{main4} and \ref{main5}}

We will use the following general properties of
tensor products of ultrafilters. 

\begin{lemma}\label{lemmatensor}
Let $X\subseteq\N^m$.
\begin{enumerate}
\item
Assume that:
$\exists\,\overline{n}_1\ \forall n_1\ge\overline{n}_1\ 
\exists\,\overline{n}_2\ \forall n_2\ge\overline{n}_2\ \ldots\
\exists\,\overline{n}_m\ \forall n_m\ge\overline{n}_m$ 
one has $(n_1,n_2,\ldots,n_m)\in X$.
Then for all non-principal ultrafilters $\U_1,\ldots,\U_m$ on $\N$,
it is $X\in\U_1\otimes\cdots\otimes\U_m$.
\item
Assume that: 
$\exists\,\overline{n}_1,N_1\ \forall n_1\ge\overline{n}_1\ 
\exists\,\overline{n}_2,N_2\ \forall n_2\ge\overline{n}_2\ \ldots\
\exists\,\overline{n}_m,N_m$ $\forall n_m\ge\overline{n}_m$ one has
$(n_1N_1,n_2N_2,\ldots,n_mN_m)\in X$.
Then for all ultrafilters $\U_1,\ldots,\U_m$ on $\N$
such that $t\N:=\{tn\mid n\in\N\}\in\U_j$ for every $j=1,\ldots,m$ and for every $t\in\N$,
it is $X\in\U_1\otimes\cdots\otimes\U_m$. 
\end{enumerate}
\end{lemma}

\begin{proof}
$(1)$.\ This proof is obtained from the proof of property $(2)$ below, 
by letting $N_j=1$ for all $j=1,\ldots,m$.

\smallskip
$(2)$.\
We proceed by induction on $m$.
In the base case $m=1$, one has $X\in\U_1$ because $X$
is a superset of $t\N$ where $t:=\overline{n}_1N_1$.
At the inductive step $m+1$, for every $a\in\N$ consider the set 
$$X(a)\ :=\ \{(a_2,\ldots,a_m,a_{m+1})\in\N^m\mid (a,a_2,\ldots,a_m,a_{m+1})\in X\}.$$
It is easily seen that for every $n_1\ge\overline{n}_1$ the 
set $X(n_1N_1)\subseteq\N^m$ satisfies the inductive hypothesis, and so
$X(n_1N_1)\in\U_2\otimes\cdots\otimes\U_m\otimes\U_{m+1}$.
But then $\{a\in\N\mid X(a)\in\U_2\otimes\cdots\otimes\U_m\otimes\U_{m+1}\}\in\U_1$,
since it includes the set $t_1\N$ where $t_1:=\overline{n}_1N_1$.
This means that $X\in\U_1\otimes(\U_2\otimes\cdots\otimes\U_m\otimes\U_{m+1})$,
as desired.
\end{proof}

%

\begin{proof}[Proof of Theorem \ref{main4}]
If $\ell_j\ne 0$ divides $k(k-1)$, then by Proposition \ref{lk},
$(\Z,\textcircled{$\star$}_{\ell_j,k_j})$
is a semigroup that has $\Z\setminus\{-\frac{k_j}{\ell_j}\}$
as a cancellative subsemigroup. So, by Lemma \ref{nonprincipalidempotents}
we can pick a \emph{non-principal} idempotent ultrafilter
$\U_j=\U_j\textcircled{$\star$}_{\ell_j,k_j}\U_j$ on $\Z$. 
If $(\ell_j,k_j)=(0,1)$, then $\textcircled{$\star$}_{0,1}$
is the sum operation on $\Z$. Clearly, $\Z\setminus\{0\}$ is a
cancellative subsemigroup of $(\Z,+)$, and also in this case
we can pick a \emph{non-principal} idempotent ultrafilter 
$\U_j=\U_j\textcircled{$\star$}_{0,1}\U_j=\U_j\oplus\U_j$ on $\Z$.
Choose the above idempotent ultrafilters 
in such a way that $\U_{j}=\U_{j'}$ whenever $(\ell_{j},k_{j})=(\ell_{j'},k_{j'})$.
Now consider the image ultrafilter $\W:=f(\U_1\otimes\cdots\otimes\U_m)$ on $\Z$.
Given a finite coloring $\Z=C_1\cup\ldots\cup C_r$, let $C_i$ be
the color such that $C_i\in\W$, and apply Corollary \ref{mt-corollary} 
where $S=T=\Z$, and where the considered associative
operations are $\ostar_j=\textcircled{$\star$}_{\ell_j,k_j}$ for $j=1,\ldots,m$.
Then for all $F_1<\ldots<F_m$ we have that
$$\left\{f\left(x^{(1)}_{F_1},\ldots,x^{(m)}_{F_m}\right)\,\Big|\, F_1<\ldots<F_m\right\}\ 
\subseteq\ C_i,$$
where if $F_j=\{n_1<\ldots<n_s\}\subset\N$, we denoted
$$x^{(j)}_{F_j}\ =\ x_{j,n_1}\textcircled{$\star$}_{\ell_j,k_j}\ldots\textcircled{$\star$}_{\ell_j,k_j} x_{j,n_s}\ =\ 
\mathfrak{S}_{\ell_j,k_j}(x^{(j)}_{n_1},\cdots,x^{(j)}_{n_s}).$$

\smallskip
Let us now turn to the case when all $\ell_j\ge 0$ and 
the function $f:\N^m\to\Z$ satisfies condition $(\dagger)$.
Pick $M\in\N$ such that $M\ge\frac{1-k_j}{\ell_j}$ for all $j$ with $\ell_j\ne 0$.
Then $\N':=\{a\in\Z\mid a\ge M\}$ is a cancellative sub-semigroup 
of $(\Z,\textcircled{$\star$}_{\ell_j,k_j})$ for every $j$.
Indeed, if $\ell_j=0$ then $k_j=1$ and $\lk$ is the sum operation,
and clearly $(\N',+)$ is a subsemigroup.
If $\ell_j\ne 0$, then notice that $M=-\frac{k_j}{\ell_j}+\frac{N_j}{\ell_j}$ where
$N_j:=\ell_j M+k_j\ge 1$ is a natural number,
and hence $(\N',\textcircled{$\star$}_{\ell_j,k_j})$ is a cancellative sub-semigroup
of $(\Z,\lk)$ by Proposition \ref{lk} (10).
Then we can pick \emph{non-principal} idempotent ultrafilters
$\U_j=\U_j\textcircled{$\star$}_{\ell_j,k_j}\U_j$ on $\N'$,
in such a way that $\U_{j}=\U_{j'}$ whenever $(\ell_{j},k_{j})=(\ell_{j'},k_{j'})$.
Now consider the tensor product $\U_1\otimes\cdots\otimes\U_m$
on $(\N')^m$, and let $\W=g(\U_1\otimes\cdots\otimes\U_m)$
be its image ultrafilter on $\Z$ under the restriction $g:=f|_{(\N')^m}:(\N')^m\to\Z$.

Without loss of generality, one can assume that in property $(\dagger)$ 
one has $\overline{n}_j\ge M$ for every $j=1,\ldots,m$, and so the set 
\begin{multline*}
X\ :=\ g^{-1}(\N) =\ \{(n_1,\ldots,n_m)\in(\N')^m\mid g(n_1,\ldots,n_m)\in\N\}\ =
\\
\{(n_1,\ldots,n_m)\in \N^m\mid n_j\ge M\ \text{for all}\ j=1,\ldots,m\ \ \text{and}\ 
f(n_1,\ldots,n_m)\in\N\}
\end{multline*}
satisfies the hypothesis of Lemma \ref{lemmatensor} (1).
Then $\N\in g(\U_1\otimes\cdots\otimes\U_m)=\W$, 
and given any finite coloring $\N=C_1\cup\ldots\cup C_r$, 
one of the colors $C_i\in\W$. We reach the thesis by
applying Corollary \ref{mt-corollary}
where $S=\N'$, $T=\Z$, $A=C_i$, and where the considered associative
operations are $\ostar_j=\textcircled{$\star$}_{\ell_j,k_j}$ for $j=1,\ldots,m$.

\end{proof}

For the next proof, we will need the existence of idempotent ultrafilters
with an additional property.

\begin{lemma}\label{lemma-tN}
Let $\ell\in\N$ and $k\in\{0,1\}$.
Then there exist idempotent ultrafilters $\U=\U\lk\U$ 
in the semigroup $(\beta\N,\lk)$ such that $t\N\in\U$ for every $t\in\N$.
\end{lemma}

\begin{proof}
Fix any $\ell\in\N$. Since $k\in\{0,1\}$, both $(\N,\textcircled{$\star$}_{\ell,0})$ and 
$(\N,\textcircled{$\star$}_{\ell,1})$ are cancellative semigroups:
this is easily checked directly, or can be derived from the general properties 
(9) and (10) of Proposition \ref{lk}.

Now 
consider the nonempty closed subspace $X:=\bigcap_{t\in\N}\mathcal{O}_{t\N}$ 
of $\beta\N$. Notice that if $a,b\in t\N$ then also $a\,\textcircled{$\star$}_{\ell,0}b=\ell ab\in t\N$
and $a\,\textcircled{$\star$}_{\ell,1}b=\ell ab+a+b\in t\N$.
In consequence, it is easily seen that $X$ is a sub-semigroup of both 
$(\beta\N,\textcircled{$\star$}_{\ell,0})$ and $(\beta\N,\textcircled{$\star$}_{\ell,1})$. 
Then, by Ellis' Lemma, there exist idempotent ultrafilters 
$\U=\U\textcircled{$\star$}_{\ell,0}\U$ in $(X,\textcircled{$\star$}_{\ell,0})$, 
and idempotent ultrafilters $\V=\V\textcircled{$\star$}_{\ell,1}\V$ in 
$(X,\textcircled{$\star$}_{\ell,1})$.
\end{proof}


\begin{proof}[Proof of Theorem \ref{main5}]
The proof is entirely similar to the above proof of Theorem \ref{main4}.
When $(\ell_j,k_j)=(0,1)$, that is, when the operation $\textcircled{$\star$}_{\ell_j,k_j}$ is
the sum on the integers, pick a non-principal ultrafilter $\U_j=\U_j\oplus\U_j$ on 
the natural numbers $\N$. It is a well-known fact
that every idempotent ultrafilter $\U$ in $(\beta\N,\oplus)$
is such that $t\N\in\U$ for every $t\in\N$ (see e.g. \cite[Lemma 5.19.1]{hs}).
When $\ell_j\in\N$ and $k_j\in\{0,1\}$, by Lemma \ref{lemma-tN}
we can pick an idempotent ultrafilter $\U_j=\U_j\textcircled{$\star$}_{\ell_j,k_j}\U_j$ on $\N$
such that $t\N\in\U$ for every $t\in\N$.
Choose the above idempotent ultrafilters 
in such a way that $\U_{j}=\U_{j'}$ whenever $(\ell_{j},k_{j})=(\ell_{j'},k_{j'})$.
Then consider the ultrafilter $\U_1\otimes\cdots\otimes\U_m$
on $\N^m$, and let $\W=f(\U_1\otimes\cdots\otimes\U_m)$
be its image ultrafilter on $\Z$ under the function $f$.
Property $(\ddagger)$ says that the set 
$$X\ :=\ f^{-1}(\N) =\ \{(n_1,\ldots,n_m)\in\N^m\mid f(n_1,\ldots,n_m)\in\N\}$$
satisfies the hypothesis of Lemma \ref{lemmatensor} (2).
So, $f^{-1}(\N)\in\U_1\otimes\cdots\otimes\U_m$, and hence $\N\in\W$. 
Then, given any finite coloring $\N=C_1\cup\ldots\cup C_r$, 
one of the colors $C_i\in\W$, and we reach the thesis by
applying Corollary \ref{mt-corollary}
where $S=T=\N$, $A=C_i$, and where the considered associative
operations are $\ostar_j=\textcircled{$\star$}_{\ell_j,k_j}$ for $j=1,\ldots,m$.
\end{proof}

\end{section}

\medskip
\section{Final remarks}

We close this paper with a list of remarks about possible
directions for future research.

\medskip
\begin{enumerate}
\item
The associative operations $\lk$ that we defined in this paper
over the integers $\Z$ also make sense in any commutative ring
$(R,+,\cdot)$, and (part of) our results could be extended
to that framework. Are there meaningful
examples that would justify such a generalization?

\smallskip
\item
In Theorems \ref{main4} and \ref{main5}
we considered polynomials in several variables with positive leading coefficient
as functions that satisfy condition $(\dagger)$ or $(\ddagger)$.
Are there are other meaningful classes of functions that satisfy those
conditions?

\smallskip
\item
The results of this paper are grounded on generalized versions
of Hindman's, Deuber's, and Milliken-Taylor's Theorems in the framework of
semigroups. Recently, also several generalizations of the Central Set Theorem
have been demonstrated for semigroups (see \cite{hi3} for a historical
survey about central sets). 
Can the study of central sets in semigroups $(\beta\N,\lk)$
lead to meaningful results in arithmetic Ramsey Theory?

\smallskip
\item
The problem of partition regularity of non-linear Diophantine equations
have been recently investigated, producing interesting results.
(Note that partition regularity of equations corresponds directly 
to finite monochromatic patterns.)
In 2017, J. Moreira \cite{mo} demonstrated that the configuration
$\{a, a+b, a\cdot b\}$ is monochromatic in the natural numbers;
this year 2021, the existence of similar monochromatic patterns,
including $\{a, a+b, a\cdot b+a+b\}$, has been proved
by J.M. Barrett, M. Lupini, and J. Moreira \cite{blm}.
It seems worth investigating
to what extent the results presented in this paper can be used
to address the general problem of partition regularity of non-linear Diophantine equations.

\smallskip
\item
If $(S,*)$ is any countable semigroup, then every bijection 
$\varphi:\N\to S$ determines an associative operation ${\textcircled{$*$}}_\varphi$
on the natural numbers by setting:
$$a\,\underset{\varphi}{\textcircled{$*$}}\,b\,=\,c\ \ \Longleftrightarrow\ \ 
\varphi(a)*\varphi(b)=\varphi(c).$$
Similar arguments to those used in this article may also be applied 
to such operations to produce partition regularity results.
In particular, operations on $\N$ induced by multiplicative subgroups
of the integers (such as the set of sums of two squares) seem worth investigating.


\smallskip
\item
A topic of research in arithmetic Ramsey Theory is about the partition regularity of infinite image partition 
regular matrices. The known examples are rather limited (see the recent paper \cite{hs2}
and references therein),
and mostly rely on Hindman Theorem. Starting from
the Finite Product Theorem for the operations $\lk$,
it may be worth investigating whether 
some new interesting classes of infinite partition regular matrices could be isolated.
\end{enumerate}

\medskip
\bibliographystyle{amsalpha}

\begin{thebibliography}{99}

\bibitem{blm}
J.M. Barrett, M. Lupini, and J. Moreira,
On Rado conditions for nonlinear Diophantine equations,
\emph{European J. Combin.} \textbf{94 }(2021), 103277.


\bibitem{be}
V. Bergelson, IP sets, dynamics, and combinatorial number theory, in
V. Bergelson, A. Blass, M. Di Nasso, R. Jin (eds.), \emph{Ultrafilters Across Mathematics},
Contemp. Math. \textbf{530}, AMS, 2010, 23--47.

\bibitem{bhw}
V. Bergelson, N. Hindman, and K. Williams,
Polynomial extensions of the Milliken-Taylor Theorem,
\emph{Trans. Amer. Math. Soc.} \textbf{366} (2014), 5727--5748.

\bibitem{bjm}
V. Bergelson, J.H. Johnson, and J. Moreira,
New polynomial and multidimensional extensions of classical partition results,
\emph{J. Combin. Theory Ser. A} \textbf{147} (2017), 119--154.

\bibitem{br}
A. Brauer, \"Uber sequenzen von potenzresten, 
\emph{Sitz.ber. Preuss. Akad. Wiss. Phys.-Math. Kl.} (1928), 9--16.

\bibitem{de}
W. Deuber, Partitionen und lineare Gleichungssysteme, 
\emph{Math. Z.} \textbf{133} (1973), 109--123.

\bibitem{dl}
M. Di Nasso and L. Luperi Baglini, 
\emph{Ramsey properties of nonlinear Diophantine equations}, 
\emph{Adv. Math.} \textbf{324} (2018), 84--117.

\bibitem{el}
R. Ellis, Distal transformation groups, \emph{Pacific J. Math.} \textbf{8} (1958), 401--405.

\bibitem{hi1}
N. Hindman, Finite sums from sequences within cells of a partition of $\N$,
\emph{J. Combin. Theory Ser. A} \textbf{17} (1974), 1--11.

\bibitem{hi2}
N. Hindman, Monochromatic sums equal to products in $\mathbb{N}$, 
\emph{Integers} \textbf{11A} (2011), article 10.

\bibitem{hi3}
N. Hindman, A history of central sets,
\emph{Ergodic Theory and Dynamical Systems} \textbf{40} (2020), 1--33.

\bibitem{hs}
N. Hindman and D. Strauss, \emph{Algebra in the Stone-\v{C}ech
Compactification, Theory and Applications} (2nd edition),
W. de Gruyter, 2011.

\bibitem{hs2}
N. Hindman and D. Strauss, 
Some new examples of infinite image partition regular matrices,
\emph{Integers} \textbf{19A} (2019), article 5.

\bibitem{ho}
J.M. Howie, \emph{Fundamentals of Semigroup Theory}, London Math. Soc. Monogr. Ser. \textbf{12},
Clarendon Press, 1995.

\bibitem{lu}
M. Lupini, Actions on semigroups and an infinitary Gowers-Hales-Jewett Ramsey Theorem,
\emph{Trans. Amer. Math. Soc.} \textbf{371} (2019), 3083--3116.

\bibitem{mi}
K. Milliken, Ramsey's Theorem with sums or unions, 
\emph{J. Comb. Theory Ser. A} \textbf{18} (1975), 276--290.

\bibitem{mo}
J. Moreira, Monochromatic sums and products in $\mathbb{N}$,
\emph{Annals of Mathematics} \textbf{185} (2017), 1069--1090.

\bibitem{ta}
A. Taylor, A canonical partition relation for finite subsets of $\omega$, 
\emph{J. Comb. Theory Ser. A} \textbf{21} (1976), 137--146.

\bibitem{VdW}
B.L. Van der Waerden, Beweis einer baudetschen vermutung, 
\emph{Nieuw Arch. Wiskd.} \textbf{15} (1927), 212--216.

\bibitem{wi}
K. Williams, Characterization of elements of polynomials in $\beta S$,
\emph{Semigroup Forum} \textbf{83} (2011), 147--160.

\end{thebibliography}

\end{document}